\documentclass[10pt]{amsart}
\usepackage{a4, amsmath, bm, amssymb}
\usepackage{setspace}
\usepackage{subfiles}
\usepackage{cite}
\usepackage{enumerate}
\usepackage{url}

\usepackage{faktor}
\usepackage{dsfont}
\usepackage{amsthm}

\makeatletter
\newtheorem*{rep@theorem}{\rep@title}
\newcommand{\newreptheorem}[2]{%
\newenvironment{rep#1}[1]{%
 \def\rep@title{#2 \ref{##1}}%
 \begin{rep@theorem}}%
 {\end{rep@theorem}}}
\makeatother

\newtheorem{theorem}{Theorem}

\newreptheorem{theorem}{Theorem}
\newtheorem{lemma}[theorem]{Lemma}
\newreptheorem{lemma}{Lemma}
\newreptheorem{claim}{Claim}
\newtheorem{proposition}[theorem]{Proposition}
\newtheorem{corollary}[theorem]{Corollary}
\theoremstyle{definition}
\newtheorem{definition}[theorem]{Definition}
\newtheorem{remark}[theorem]{Remark}

\numberwithin{equation}{section}
\numberwithin{theorem}{section}

\begin{document}

\title[Anosov flows, growth rates on covers and group extensions]{Anosov flows, growth rates on covers and group extensions of subshifts}

\author{Rhiannon \textsc{Dougall} and Richard \textsc{Sharp}}

\address{School of Mathematics, University of Bristol, Bristol, BS8 1TW, U.K.}
\email{R.Dougall@bristol.ac.uk}

\address{Mathematics Institute, University of Warwick,
Coventry CV4 7AL, U.K.}
\email{R.J.Sharp@warwick.ac.uk}

\begin{abstract} 
The aim of this paper is to study growth properties of group extensions of hyperbolic dynamical systems, where we do not assume that the extension satisfies the symmetry conditions seen, for example, 
in the work of Stadlbauer on symmetric group extensions and of the authors on geodesic flows.
Our main application is to growth rates of periodic orbits for covers of an Anosov flow:
we reduce the problem of counting periodic orbits in an amenable cover $X$ to counting in a maximal abelian subcover $X^{\mathrm{ab}}$. In this way, we obtain an equivalence for the Gurevi\v{c} entropy: $h(X)=h(X^{\mathrm{ab}})$ if and only if the covering group is amenable.
In addition, when we project the periodic orbits for amenable covers $X$ to the compact factor $M$, they equidistribute with respect to a natural equilibrium measure -- in the case of the geodesic flow, the measure of maximal entropy.
\end{abstract}

\maketitle


\normalsize

\section{Introduction}

The aim of this paper is to study growth properties of group extensions of hyperbolic dynamical systems, where we do not assume that the extension satisfies the symmetry conditions imposed in \cite{DougallSharp} and \cite{Stadlbauer13}, for example.  
Our main application is to growth rates of periodic orbits for covers of Anosov flows
and we obtain generalisations of results previously known for geodesic flows over compact (or even convex
co-compact) negatively curved manifolds \cite{DougallSharp}, \cite{Roblin}, \cite{Stadlbauer13}.
We begin by describing these results.

Let $M$ be a compact smooth Riemannian manifold and let $\phi^t : M \to M$ be a transitive 
Anosov flow.
Then $\phi^t$ has a countable set of periodic orbits $\mathcal P(\phi)$ and, for 
$\gamma \in \mathcal P(\phi)$,
we write $l(\gamma)$ for its period. It is well-known that the growth rate of periodic orbits is given by the topological entropy of $\phi^t$; more precisely
\[
\lim_{T \to \infty} \frac{1}{T} \log \#\{\gamma \in \mathcal P(\phi) \hbox{ : } l(\gamma) \le T\} =: h
= h_{\mathrm{top}}(\phi).
\]

Now suppose that $X$ is a regular cover of $M$ with covering group $G$, i.e. $G$ acts 
freely and isometrically on $X$ such that $M = X/G$.
Let $\phi_X^t : X \to X$ be the lifted flow, which we assume to be transitive.
We will be interested in the growth of periodic orbits for $\phi_X^t$.
If $G$ is finite then $\phi_X^t$ is also an Anosov flow and $h_{\mathrm{top}}(\phi_X)=
h_{\mathrm{top}}(\phi)$.
If $G$ is infinite the situation is more interesting.
First, note that if $\phi_X^t$ has a periodic orbit $\gamma$ then the translates of
 $\gamma$ by the action of $G$ give infinitely many periodic orbits with the same period,
 so a naive definition of periodic orbit growth does not make sense. Rather, we follow the approach of 
\cite{PaulinPollicottSchapira15} and, choosing an open, relatively compact set $W \subset X$, 
define 
\[
h(X) := \limsup_{T \to \infty} \frac{1}{T} \log \#\{\gamma \in \mathcal P(\phi_X) \hbox{ : }
l(\gamma) \le T, \ \gamma \cap W \neq \varnothing\}.
\]
As the notation suggests, $h(X)$ is independent of the choice of $W$ (see Lemma \ref{gurevich}).
The above definition is analogous to that made by Gurevi\v{c} for subshifts of finite 
type \cite{Gurevic} (see below) and it is natural to call $h(X)$ the Gurevi\v{c} entropy of $\phi_X^t$.
It is easy to see that $h(X) \le h$.

Let us now restrict to the special case that $M=SV$ is the unit-tangent bundle over a compact manifold $V$ with 
negative sectional curvatures and that $\phi^t$ is the geodesic flow. (Notice that this flow admits a time-reversing involution $\iota : SV \to SV$ defined by $\iota(x,v) = (x,-v)$, where $x \in V$ and $v \in S_xV$.)
Combining the results of \cite{Roblin} and \cite{DougallSharp}, we have $h(X) = h$ if and only if $G$ is amenable.
(Recall that $G$ is amenable if it has a \emph{Banach mean}, that is, there is a $G$-invariant bounded linear functional $\mathfrak M : \ell^\infty(G,\mathbb R) \to \mathbb R$ such that, for $x \in \ell^\infty(G,\mathbb R)$,
$\inf_{g \in G} x(g) \le \mathfrak M(x) \le \sup_{g \in G} x(g)$.)
This equivalence fails to hold for general Anosov flows -- however, the results of this paper will say that, for an amenable cover, the counting is reduced to the maximal abelian subcover.
This will be discussed in greater detail in section \ref{anosovflows} but we note at this point that if $Y$ is an abelian cover of $M$ then $h(Y) =h$ if and only if $\Phi^Y_{\mu_0}=0$, 
where 
$\Phi^Y_{\mu_0}$ is the (relative) winding cycle associated to the measure of maximal entropy for $\phi^t$.
Our main result is that the correct comparison is between $h(X)$ and $h(X^{\mathrm{ab}})$,
where $X^{\mathrm{ab}}$ is the maximal abelian subcover of the cover $X \to M$
(so that the covering group for 
$X^{\mathrm{ab}} \to M$ is $G^{\mathrm{ab}} = G/[G:G]$, the abelianization of $G$).
Clearly, $h(X) \le h(X^{\mathrm{ab}})$ and we have the following theorem.

\begin{theorem} \label{intro-main}
We have $h(X) = h(X^{\mathrm{ab}})$ if and only if $G$ is amenable.
\end{theorem}

An immediate consequence of this result and the characterisation of $h(X^{\mathrm{ab}})$ in Proposition 
\ref{asymptotics_abeliancover} is the following corollary, which expresses $h(X)$ in terms of 
measure-theoretic entropies of 
$\phi$-invariant measures.
Let $\mathcal M(\phi)$ denote the set of $\phi$-invariant probability measures on $M$ and,
for $\mu \in \mathcal M(\phi)$, let $h_\phi(\mu)$ denote the measure-theoretic entropy of 
$\phi$ with respect to $\mu$.

\begin{corollary} \label{corollary-intro-main}
We have $h(X)=\sup\left\{ h_\phi(\mu) : \mu \in \mathcal M(\phi),
\ \Phi^{X^{\mathrm{ab}}}_\mu = 0 \right\}$ if and only if $G$ is amenable.
\end{corollary}

Readers more familiar with geodesic flows may wonder at this point, when is $\pi_1(M)$ non-amenable? In general it is very difficult to say anything about the topology of $M$ as the flow lines may not be efficient for measuring the fundamental group (for a discussion in dimension 3 see \cite{Fenley}, \cite{BartelmeFenley}). Nevertheless, our method is able to relate orbits and topology for \textit{transitive} covers.

The proof of Theorem \ref{intro-main} will use symbolic dynamics and, in particular, group extensions of 
subshifts of finite type. As part of the our approach we will obtain results in this seting which
are of independent interest. 
(Some definitions are deferred until section \ref{section:ssftge}.)
Let $\sigma : \Sigma \to \Sigma$ be a mixing subshift of finite type.
We will write $h(\sigma)$ for the topological entropy of $\sigma$ and note that this is equal to the 
exponential growth rate of periodic points,
\[
h(\sigma) = \lim_{n \to \infty} \frac{1}{n} \log \#\{x \in \Sigma \hbox{ : } \sigma^nx=x\}.
\]
For a countable group $G$ and a function $\psi : \Sigma \to G$, 
we consider the skew-product dynamical system
\[
T_\psi : \Sigma\times G \to \Sigma\times G
:
(x,g)\mapsto (\sigma x, g\psi(x)).
\]
We say that 
$T_\psi : \Sigma \times G \to \Sigma \times G$
is a \emph{$G$-extension} of $\sigma : \Sigma \to \Sigma$. 
We will always assume that $T_\psi : \Sigma\times G \to \Sigma\times G$ is transitive.
Let us consider the periodic points of $T_\psi$. Clearly,
$T_\psi^n(x,g) = (x,g)$ if and only if 
$\sigma^n x=x$ and $\psi^n(x)=e$, the identity in $G$.
If $G$ is infinite then $T_\psi^n$ has infinitely many periodic points of fixed period $n$;
however, we can use the following definition, due to Gurevi\v{c} \cite{Gurevic}, to obtain a growth
rate. We say that the {\it Gurevi\v{c} entropy} of $T_\psi$ is
\[
h_{\mathrm{Gur}}(T_\psi) = \limsup_{n\to \infty} \frac{1}{n}\log 
\#\{x \in \Sigma \hbox{ : } \sigma^nx=x, \ \psi^n(x)=e\}.
\]
Clearly, $h_{\mathrm{Gur}}(T_\psi) \le h(\sigma)$.
However,
it is easy to construct examples where $G$ is amenable but
$h_{\mathrm{Gur}}(T_\psi) < h(\sigma)$.
For example, let $\Sigma = \{0,1,2\}^{\mathbb Z}$, $G = \mathbb Z$
and define $\psi$ by $\psi(x)=\psi(x_0)$ with $\psi(0)=\psi(1)=1$ and $\psi(2)=-1$.
Then $T_\psi$ is transitive.
By \cite{PollicottSharp94}, 
\[
h_{\mathrm{Gur}}(T_\psi) = \sup\left\{h_\sigma(m) \hbox{ : } m \in \mathcal M(\sigma), \ \int \psi \, dm=0\right\}.
\]
Now, the measure of maximal entropy $m_0$ for $\sigma : \Sigma \to \Sigma$ is the 
$(1/3,1/3,1/3)$-Bernoulli measure and, clearly,
$\int \psi \, dm_0 = 1/3 \ne 0$. Thus, by the Variational Principle,
$h_{\mathrm{Gur}}(T_\psi) < h(\sigma)$.

We shall show that, in fact, the natural comparison is between
 $h_{\mathrm{Gur}}(T_\psi)$ and $h_{\mathrm{Gur}}(T_{\psi^{\mathrm{ab}}})$,
where $T_{\psi^{\mathrm{ab}}} : \Sigma \times G^{\mathrm{ab}} \to \Sigma \times G^{\mathrm{ab}}$
is the induced $G^{\mathrm{ab}}$-extension, where $G^{\mathrm{ab}} = G/[G,G]$ is the
abelianization of $G$. More precisely, if $\pi : G \to G^{\mathrm{ab}}$ is the natural projection then
$\psi^{\mathrm{ab}} = \pi \circ \psi$.

\begin{theorem} \label{intro-gurevicentropy}
If $T_\psi$ is transitive then $h_{\mathrm{Gur}}(T_\psi) =h_{\mathrm{Gur}}(T_{\psi^{\mathrm{ab}}})$
if and only if $G$ is amenable.
\end{theorem}

We now put these results into a broader context. This discussion will focus on more recent viewpoints, and on the role of symmetry, and so neglects to mention the spectral analogues of Brooks \cite{Brooks}, \cite{Brooks85}, whose work was a motivating factor in the recent developments of this field (for instance \cite{Stadlbauer13}, \cite{DougallSharp}, \cite{Dougall}, \cite{CDS18}, \cite{CDST}).

We have already given the definition of an amenable group in terms of the existence of a Banach mean. We give an equivalent criterion of amenability due to Kesten \cite{Kesten}, which is often taken as the definition for amenability. Throughout, let $p$ be a probability measure on $G$, which we will assume to be finitely supported and for the support to generate $G$. We say that $p$ is \emph{symmetric} if $p(g)=p(g^{-1})$ for all $g\in G$.

Kesten's theorem says that, for a \emph{symmetric} probability on $G$, the decay of the probability to return to the identity is subexponential if and only if $G$ is amenable. More precisely, writing $\lambda(p) = \limsup_{n\to\infty} (p^{*n}(e))^{1/n}$, where $p^{*n}$ denotes the $n$th convolution of $p$, we have $\lambda(p)=1$ if and only if $G$ is amenable. This result has an application to counting for normal subgroups of non-elementary Gromov hyperbolic groups. Let $\Gamma$ be a non-elementary Gromov hyperbolic group and $\Gamma^\prime$ a normal subgroup. The uniform measure on the ball of radius $n$ in $\Gamma$ descends to a probability measure $p_n$ on $G=\Gamma/\Gamma^\prime$. It can be shown, for example in \cite{CDS18}, that $\limsup_{n\to\infty}\log \lambda(p_n) \le \delta_{\Gamma^\prime}-\delta_{\Gamma}\le 0$; and so in the case that $G$ is amenable we have equality between the critical exponents $\delta_{\Gamma^\prime}$, $\delta_{\Gamma}$. (We also mention the more general approach of Roblin \cite{Roblin} who uses the Banach mean property to prove the result only for normal subgroups $\Gamma^\prime$ of $\Gamma$ a non-elementary discrete group of isometries of a simply connected negatively curved manifold -- notably there is no compactness assumption on $\Gamma$.)
Because of this avenue of approximating by a sequence of probability measures $p_n$, the community has sometimes referred to statements ``$G$ amenable implies $\cdots$ " as ``the easy direction". However, it is clear that the aforementioned application crucially used the existence of a natural family of symmetric probability measures. As we will see from the literature reviewed in subsection \ref{subsection:windingcycles}, extending such orbital counting results beyond the inherent symmetry of the isometric actions to Anosov flows is highly non-trivial.

The case that $G$ is non-amenable is more robust in the absence of symmetry. In the case of random walks, it can happen that $G$ is amenable, and $\lambda(p)<1$. The Kesten criterion was generalised by Day \cite{Day} to a criterion on the $\ell^2(G)$ spectral radius $\mathrm{spr}(M_p)$ of the random walk operator $M_p f(x)=\sum_g p(g)f(xg)$. In this way, $G$ is amenable if and only if $\mathrm{spr}(M_p)=1$. Observe that $\lambda(p)\le \mathrm{spr}(M_p)$ and so one direction of the Kesten result is contained in this.
In the setting of subshifts of finite type, it is natural to consider the transfer operator $\mathcal{L}_0$ (for a full exposition see section \ref{section:gege}). Stadlbauer \cite{Stadlbauer13} showed that $\mathrm{spr}(\mathcal{L}_0)=\exp(h(\sigma))$ if $G$ amenable; and Jaerisch \cite{Jaerisch} showed the converse. Analogous to the case of
random walks, the growth quantity we wish to estimate for subshifts of finite type satisfies 
$\exp(h_{\mathrm{Gur}}(T_\psi))\le \mathrm{spr}(\mathcal{L}_0)$.

We conclude the introduction by outlining the contents of the rest of the paper.
In section \ref{anosovflows} we introduce Anosov flows and their lift to covers. In particular, we give an 
account of results about growth of periodic orbits for abelian covers.
In section \ref{section:ssftge}, we discuss subshifts of finite type and introduce the notion of Gurevi\v{c} 
pressure.
In section \ref{section:gege}, we discuss Gurevi\v{c} entropy for group extensions of subshifts of finite type
and prove one direction in Theorem \ref{intro-gurevicentropy}.
The proof of Theorem \ref{intro-gurevicentropy} is completed in section \ref{amenabilityandpressure},
which contains the key to obtaining results in the absence of symmetry.
In section \ref{symbolicdynamics}, we return to Anosov flows and discuss how they 
may be coded in terms of subshifts of finite type and section \ref{coversandgroupextensions} extends this to covers and group extensions.
In section \ref{proofofmaintheorem}, we complete the proof of Theorem \ref{intro-main}. Finally,
in section \ref{sectiononequidistribution}, we state and proof a number of equidistribution results.

\section{Anosov Flows} \label{anosovflows}

\subsection{Anosov flows, periodic orbits and pressure}
Let $\phi^t:M\to M$ be an Anosov flow, 
i.e. that the tangent bundle has a continuous $D\phi$-invariant splitting $TM = E^0 \oplus E^s \oplus E^u$,
where $E^0$ is the one-dimensional bundle tangent to the flow and where there exist constants $C,\lambda >0$ such that
\begin{enumerate}
\item
$\|D\phi^tv\| \le Ce^{-\lambda t} \|v\|$, for all $v \in E^s$ and $t >0$;
\item
$\|D\phi^{-t}v\| \le Ce^{-\lambda t} \|v\|$, for all $v \in E^u$ and $t>0$.
\end{enumerate}
In addition, we assume that $\phi^t : M \to M$ is transitive and weak mixing.

For some intuition behind Anosov flows, one should note the following constructions (and we will make use of them soon).
An $\epsilon$-pseudo-orbit for $\phi^t$ is a path $\tau:I\to X$ (where $I$ is an interval)
so that for all $t,t+\delta\in I$, $|\delta|<1$,
$$
d(\tau(t+\delta),\phi^\delta(\tau(t)))<\epsilon.
$$
We say that $\tau:\mathbb{R}\to X$ is a periodic $\epsilon$-pseudo-orbit if it is a periodic map of $\mathbb{R}$, and satisfies $d(\tau(t+\delta),\phi^\delta(\tau(t)))<\epsilon$ for all $t,t+\delta\in \mathbb{R}$, $|\delta|<1$.

\begin{lemma}[Anosov Closing Lemma, \cite{FH}, Theorem 5.3.10]\label{anosovclosinglemma}
Let $\Lambda$ be a hyperbolic set for a flow $\phi^t$. Then there exists a neighbourhood $U$ of $\Lambda$ and numbers $\epsilon_0,L>0$ such that for $\epsilon\le \epsilon_0$, any periodic $\epsilon$-pseudo-orbit is $L\epsilon$-shadowed by a unique periodic orbit for $\phi^t$.
\end{lemma}

Write $\mathcal P(\phi)$ for the set of periodic orbits of $\phi$ and, for $\gamma \in \mathcal P(\phi)$,
write $l(\gamma)$ for its period. The limit
\[
h= \lim_{T\to\infty} \frac{1}{T}\log \# \left\{\gamma \in \mathcal P(\phi): l(\gamma)\le T \right\}
\]
exists and is positive, and $h$ is also equal to the topological entropy of $\phi$. We also have the variational
principle 
\[
h=\sup \left\{ h_\phi(\mu): \mu\in \mathcal{M}(\phi)\right\},
\]
where $h_\phi(\mu)$ denotes the measure-theoretic entropy of $\phi$ with respect to $\mu$, 
and $\mathcal{M}(\phi)$ is the collection of $\phi$-invariant Borel probability measures on $M$.
The supremum is attained at a unique measure $\mu_0$, called the measure of maximal entropy for $\phi$.
We also have a weighted version of this set-up. For a continuous function $F : M \to \mathbb R$, 
we can define the pressure
$P(F,\phi)$ by
\[
P(F,\phi) = \lim_{T \to \infty} \frac{1}{T} \log \sum_{\substack{\gamma \in \mathcal P(\phi):
\\ l(\gamma) \le T}} \exp \int_\gamma F
\, dt,
\]
where 
\[
\int_\gamma F := \int_0^{l(\gamma)} F(\phi_t x_\gamma)
\, dt
\]
with
$x_\gamma \in \gamma$.
In this case, we also have a variational principle,
\[
P(F,\phi) = \sup\left\{h_\phi(\mu) + \int_M  F \, d\mu \hbox{ : } 
\mu \in \mathcal M(\phi)\right\}.
\]
The supremum is attained by a unique measure $\mu_F$ which is ergodic; we call $\mu_F$ the equilibrium state for $F$. In section \ref{sectiononequidistribution}, we will use the following lemma (see Theorem 8.2
and Theorem 9.12
of \cite{Walters}).

\begin{lemma} \label{entropyusc}
The map $\mathcal M(\phi) \to \mathbb R : \mu \mapsto h_\phi(\mu)$ is upper semi-continuous
and
\[
h_\phi(\mu) = \inf \left\{P(F,\phi) - \int F \, d\mu \hbox{ : } F \in C(M,\mathbb R)\right\}.
\]
\end{lemma}

For H\"older continuous functions
$F,G : M \to \mathbb R$, the function
$\mathbb R \to \mathbb R : t \mapsto P(F+tG)$ is real-analytic and
\begin{equation}
\frac{dP(F+tG)}{dt}\bigg|_{t=0} = \int_M  G \, d\mu_F.
\end{equation}

\subsection{Covers}
Suppose that 
$X$ is a regular cover of $M$ with covering group $G$.
Let $\phi_X^t:X\to X$ be the lift of $\phi$ to $X$.
In particular, the $G$ action commutes with $\phi_X^t$. 
Write $\mathcal P(\phi_X)$ for the set of periodic orbits of $\phi_X$ and, as above, write $l(\gamma)$ for the 
period of the periodic orbit $\gamma$.
Fix an open set $W$ in $M$ whose closure is compact. Write 
\begin{equation} \label{def of X counting function}
\pi_X(T,W) = \#\left\{ \gamma \in \mathcal P(\phi_X) : l(\gamma) \le T 
\ \mathrm{and} \ \gamma\cap W \ne \varnothing\right\}
\end{equation}
and
\begin{equation} \label{def of h(X)}
h(X)=\limsup_{T\to\infty}\frac{1}{T}\log\pi_X(T,W)
\end{equation}
 for the exponential growth rate of $\pi_X(T,W)$.

We also introduce another counting function associated to the cover $X$. 
Write $\Pi_X(T)$ for 
the cardinality of the set of periodic orbits for $\phi$ in $M$ of period at most $T$, and 
which lift to a periodic orbit for $\phi_X$.

We have the following lemma.

\begin{lemma} \label{gurevich}
The value of $h(X)$ is independent of the choice of $W$.

In addition,
$$
h(X) = \limsup_{T\to\infty}\frac{1}{T}\log \Pi_X(T).
$$
\end{lemma}

\begin{proof}
Let us write $h(X,W)=\limsup_{T\to\infty}\frac{1}{T}\log\pi_X(T,W)$ to emphasise the choice of $W$.
We begin by showing that $h(X,W)$ is independent of the choice of $W$.

First, we observe that it is sufficient to consider ``small'' $W$. To see this, lift our preferred metric on $M$ to $X$. Since the closure of $W$ is compact, it has bounded diameter $R$. For every $\epsilon>0$ we can find a finite family of open sets of diameter less than $\epsilon$ that cover $W$. Then it is easy to see that there is $W_\epsilon$ of diameter less than $\epsilon$ with $h(X,W)\le h(X,W_\epsilon)$. On the other hand, since $W$ is open it must contain an open set $W_\delta$ of
diameter less that $\delta$ for all sufficiently small $\delta$. In this way we have $h(X,W_\delta)\le h(X,W)\le h(X,W_\epsilon)$.

Let $W_1,W_2$ be two sets of sufficiently small diameter 
$\mathrm{diam}(W_1),\mathrm{diam}(W_2)<\epsilon_0$, where $\epsilon_0$ is given in by the 
Anosov Closing Lemma 
(Lemma \ref{anosovclosinglemma}). We will also assume that $\epsilon_0$ is sufficiently small that an 
open ball in $M$ of this diameter is simply connected. 

We want to show that $h(X,W_1)\le h(X,W_2)$.

As $W_2$ is open it contains a ball $B_2$ of radius some $\delta$ and a $\lambda\delta$ sub-ball, where 
$\lambda$ is chosen so that $\lambda\delta + L\lambda\delta<\delta$
and $L$ is also given by the Anosov Closing Lemma. Write $\delta_2=\lambda\delta$
(and note that $\delta_2 \le \epsilon_0$). Cover $W_1$ by 
finitely many $\delta_2$ balls $B_1^i$ (where $i$ runs over some finite indexing set).
 For each $i$, fix an orbit $c_i^1$ from $B_1^i\cap W_1$ to $B_2$. Fix orbits $c_i^2$ from $W_2$ 
 to $B_1^i\cap W_1$. This is possible by transitivity.
Let $\tau_1$ be a periodic $\phi_X$-orbit intersecting $W_1$ and write $T_1$ for its period. There is some $i$ so that $\tau_1$ is $\delta_2$ close to $c_i^1$ and $c_i^2$. Then $c_i^1\tau_1c_i^2$ is a 
$\delta_2$-pseudo-orbit through $B_2$.

Now consider the projection of $c_i^1\tau_1c_i^2$ to $M$. By the Anosov Closing Lemma,
this is $L\delta_2$-shadowed by a
unique periodic orbit $\gamma_0$. 
Since $L \delta_2 < \delta_2 < \epsilon_2$, $\gamma$ lifts to a $\phi_X$-periodic orbit $\gamma$ on $X$
which $L\delta_2$ shadows $c_i^1\tau_1c_i^2$.
Unpicking the value of $\delta_2$, gives that $\gamma$ passes within distance $L\lambda\delta$ of the 
$\lambda\delta$ sub-ball. Since $\lambda\delta + L\lambda\delta<\delta$, we conclude that $\gamma$ passes through $B_2$, and therefore $W_2$. In addition, it is clear (for $T_1$ sufficiently large) that the period $T_{\gamma}$ of $\gamma$ is bounded by $|T_{\gamma}- T_1| \le c^1_{i} + c^2_{i}+L\delta_2$.

We have thus created a map from the set of $\phi_X$-periodic orbits intersecting $W_1$
to the set of $\phi_X$-periodic orbits intersecting $W_2$-periodic orbits, that distorts periods by an
additive constant. Observe that when the period of $\tau_1$ is sufficiently large, the mapping is an injection (otherwise we would have distinct periodic orbits shadowing each other).

We have concluded that $h(X,W_1)\le h(X,W_2)$. Since $W_1,W_2$ were arbitrary, this gives $h(X,W_1)=h(X,W_2)$ for all $W_1,W_2$, as required.

We finish be showing the final part of the lemma. In general, it is clear that
$$
h(X)\ge \limsup_{T\to\infty}\frac{1}{T}\log N_{X}(T)
$$
(since we can choose $W$ to be a fundamental domain for the covering and so each $\gamma$ counted by $\Pi_X(T)$ will have at least one left intersecting $W$).
To see the reverse inequality, we simply choose $W$ to be a small ball in $Y$. 
Let
\[
C = \inf\{t >0 \hbox{ : } \exists x \in W \ \mathrm{such} \  \mathrm{that} \ \phi_X^tx \in GW \setminus W\},
\]
then $C>0$. 

Let $\gamma$ be a $\phi_X$-periodic orbit of period $\le T$ intersecting $W$, and let $\gamma_0$ be its projection to $M$. Fixing $\gamma_0$, there are at most $T/C$ such $\gamma$. In this way we deduce that $\Pi_{X}(T)\ge (C/T)\pi_X(T,W)$, and so $h(X)=\limsup_{T\to\infty} T^{-1} \log \Pi_{X}(T)$.

\end{proof}
 
\subsection{Abelian covers, relative winding cycles and counting} \label{subsection:windingcycles}

We will now specialise to abelian covers of $M$; to emphasise this distinction, we shall denote the covering space by $Y$.
Let $Y$ be a regular cover of $M$, with {\it abelian} covering group $G$. (Since $G$ is abelian, we will use $0$ to denote its identity element.)
We begin by defining winding cycles relative to this cover. First note that we may write
$G = A \oplus F$, where $A$ is a free abelian group of rank $a \ge 0$ and $F$ is the finite 
subgroup of torsion elements. 
If $a=0$ then the lifted flow $\phi_Y^t : Y \to Y$ is also an Anosov flow and it has topological entropy equal to 
$h$, so there is nothing interesting to say in this case. We therefore suppose that $a \ge 1$. It is best to view $A$ as a free 
$\mathbb Z$-module. 

The cover $Y$ is itself covered by the universal abelian cover $\overline Y$ of $M$ and there is a natural surjective homomorphism from $H_1(M,\mathbb Z)$ to $G$ which (by factoring out $F$) induces a surjective 
module homomorphism
$\alpha : H \to A$, where
$H =  H_1(M,\mathbb Z)/\mathrm{torsion}$
is a free $\mathbb Z$-module of rank $b \geq a$.
 (Of course, $H$ is also identified as an integral lattice in 
$H_1(M,\mathbb R)$.) We may then write
$H = A \oplus \ker \alpha$.

Let $c_1,\ldots, c_a$ be a basis for $A$ (which we may identify with elements of $H_1(M,\mathbb R)$)
and (if $b >a$) let $c_{a+1},\ldots, c_b \in \ker \alpha$ extend this to a basis for $H$.
Then $\{c_1,\ldots,c_a\}$ spans a vector subspace of $H_1(M,\mathbb Z)$ which may be identified
with $A \otimes_{\mathbb Z} \mathbb R$. We shall denote this vector space by $\Gamma(M,Y)$. 

Now let $\omega_1,\ldots,\omega_b$ be closed $1$-forms on $M$ representing
cohomology classes $[\omega_1],\ldots,[\omega_b]$ in $H^1(M,\mathbb R)$ which are dual to
$c_1,\ldots,c_b$, i.e.
\[
\langle [\omega_i], c_j \rangle := \int_{c_j} \omega_i = \delta_{ij}.
\]
(Here, $\delta_{ij}$ denotes the Kronecker symbol.) We can identify the span of
$\{[\omega_1],\ldots,[\omega_a]\}$ with the dual space of $\Gamma(M,Y)$, which we denote by
$\Gamma(M,Y)^*$. (If $Y = \overline Y$ then
$\Gamma(M,Y) = H_1(M,\mathbb R)$ and $\Gamma(M,Y)^* = H^1(M,\mathbb R)$.)

Now consider a measure $\mu \in \mathcal M(\phi)$. We define an element $\Phi^Y_\mu \in 
\Gamma(M,Y)$,
called the winding cycle of $\mu$ relative to the cover $Y$, by its action on the dual space
$\Gamma(M,Y)^*$.
For $[\omega_i]$, $i=1,\ldots,a$, we define
\[
\langle \Phi_\mu^Y, [\omega_i] \rangle = \int_M \omega_i(\mathcal X_\phi) \, d\mu
\]
and extended by linearity, where $\mathcal X_\phi$ is the vector field generating $\phi$. 
(If $df$ is an exact $1$-form then $df(\mathcal X_\phi)$ is the derivative of $f$ in the flow direction,
i.e. $df(\mathcal X_\phi)(x) = \lim_{t \to 0+} t^{-1}(f(\phi^tx)-f(x))$. Hence, since $\mu$ is $\phi$-invariant,
$\int_M df(\mathcal X_\phi) \, d\mu =0$, so $\Phi_\mu^Y$ is well-defined.)

Associated to each periodic orbit $\gamma \in \mathcal P(\phi)$, there is an element $[\gamma] \in G$,
defined as follows. Let $\gamma = \{\phi^tx \hbox{ : } 0 \le t \le l(\gamma)\}$ and
let $\widetilde x$ be a lift of $x$ to $Y$. Then $\phi_Y^{l(\gamma)} \widetilde x = g\widetilde x$, for some 
$g \in G$. Since $G$ is abelian, $g$ is independent of the choice of lift and we define $[\gamma]=g$.
Notice that $\gamma$ lifts to a periodic orbit if and only if $[\gamma]=0$. 
We say that $\phi^t : M \to M$ is {\it $Y$-full} if 
$\{[\gamma] \hbox{ : } \gamma \in \mathcal P(\phi)\} = G$. (This generalises the notion of homologically full, introduced in \cite{Sharp93}, for the case where $Y$ is the universal abelian cover of $M$.) We will show in section \ref{symbolicdynamics} that if $\phi_Y^t : Y \to Y$ is transitive then $\phi$ is $Y$-full.

The counting function $\Pi_Y(T)$ introduced above may be written as
 $\Pi_Y(T) = \#\{\gamma \in \mathcal P(\phi) \hbox{ : } [\gamma]=0\}$.
The following result was proved in \cite{Sharp93} in the case where $Y$ is the universal abelian cover and
$G = H_1(M,\mathbb Z)$ but the proof immediately extends to arbitrary abelian covers. The results in 
\cite{Sharp93} are also phrased in terms of {\it prime} periodic orbits but it is easy to see that the number of non-prime periodic $\phi$-orbits that left to a periodic orbit on $Y$ is of order
$O(T e^{h(Y)T/2})$ and so does not affect the asymptotic (see Lemma \ref{sameabscissa} below).
 
\begin{proposition} [Sharp \cite{Sharp93}] \label{asymptotics_abeliancover} 
The following statements are equivalent:
\begin{enumerate}
\item[\rm{(i)}]
the set $\{[\gamma]_{\mathrm{TF}} \hbox{ : } \gamma \in \mathcal P(\phi)\}$ is not contained in a closed 
half-space of $\mathbb R^a$,
where $[\gamma]_{\mathrm{TF}} \in \mathbb R^a$ is the torsion-free part of $[\gamma]$;
\item[\rm{(ii)}]
$\phi : M \to M$ is $Y$-full; 
\item[\rm{(iii)}] 
there exist a constant $C>0$ such that
\[
\Pi_Y(T) \sim C\frac{e^{h(Y)T}}{T^{1+a/2}}, \ \mathrm{as} \ T \to \infty.
\]
\end{enumerate}
Furthermore, 
\[
h(Y) = \sup\left\{h(\mu) \hbox{ : } \mu \in \mathcal M(\phi), \ \Phi_\mu^Y=0\right\}.
\]
\end{proposition} 

In view of the uniqueness of the measure of maximal entropy, we immediately have the following corollary.

\begin{corollary}
We have $h(Y) = h$ if and only if $\Phi_{\mu_0}^Y=0$, where $\mu_0$ is the measure of maximal 
entropy for $\phi$.
\end{corollary}
 
 \begin{remark}
(i) A similar asymptotic holds for $\#\{\gamma \in \mathcal P(\phi) \hbox{ : } l(\gamma) \le T, \ [\gamma]=\alpha\}$
 for any $\alpha \in G$, the only modification being that the constant $C$ is changed to 
 $Ce^{\langle \xi,\alpha'\rangle}$, where $\alpha'$ is the torsion-free part of $\alpha$ and $\xi$ is a certain 
 cohomology class in $H^1(M,\mathbb R)$. In fact, $\xi =0$ if and only if $\Phi_{\mu_0}^Y = 0$.
 
 \noindent
 (ii) Earlier results were obtained for geodesic flows over compact negatively curved manifolds:
 Phillips and Sarnak\cite{PhillipsSarnak87} (constant curvature manifolds), Katsuda and Sunada
 \cite{KatsudaSunada88} (constant curvature surfaces), Lalley \cite{Lalley89} and Pollicott
 \cite{Pollicott91} (variable curvature surfaces).
 All these results exploited the time-reversal symmetry of the geodesic flow. 
The extension to Anosov flows was made by Katsuda and Sunada \cite{KatsudaSunada90} under the assumption that the winding cycle for the measure of maximal entropy vanishes.
Results giving more detailed information about the asymptotic behaviour are contained in 
\cite{Anantharaman00}, \cite{Kotani01}, \cite{PollicottSharp01} and \cite{Sharp04}.
 \end{remark}
 
 We can relate the growth rate $h(Y)$ to pressure in the following way \cite{Sharp93}. For each 
 $i=1,\ldots,a$, define
 a H\"older continuous function $\Psi_i : M \to \mathbb R$ by
 $\Psi_i = \omega_i(\mathcal X_\phi)$, where $\omega_i$ and $\mathcal X_\phi$ are as above. Now 
 write $\Psi = (\Psi_1,\ldots,\Psi_a)$ and define $\beta : \mathbb R^a \to \mathbb R$ by
 \[
 \beta(w) = P(\langle w,\Psi\rangle,\phi),
 \]
 where $w =(w_1,\ldots,w_a)$ and $\langle w,\Psi \rangle = \sum_{i=1}^a w_i\Psi_i$.
  If $\phi$ is $Y$-full then the function $\beta$ is strictly convex and there exists a unique 
  $\xi \in \mathbb R^a$ for which $\nabla \beta(\xi) =0$. We then have 
    \[
  h(Y) = \beta(\xi) = h_\phi(\mu_{\langle \xi,\Psi \rangle}).
  \]
 
 Let us summarise this in a lemma.
 
 \begin{lemma} [Sharp \cite{Sharp93}] \label{zeromean}
If $\phi^t : M \to M$ is $Y$-full then there exists a unique $\xi \in \mathbb R^a$ such that
    \[
  h(Y) = \beta(\xi) = h_\phi(\mu_{\langle \xi,\Psi \rangle}).
  \]

 \end{lemma}
 
Another characterisation of $h(Y)$ is as the abscissa of convergence of
the series
 \[
 \sum_{\substack{\gamma \in \mathcal P(\phi) : \\ [\gamma]=0}}
 e^{-sl(\gamma)}.
\]
Later, it will be convenient to replace this series with a modified version.
We introduce some notation.
Let $\mathcal P(\phi)'$ denote the set of prime periodic orbits.
For $\gamma \in \mathcal P(\phi)$, we define $\Lambda(\gamma)$ as follows. 
Any we may write $\gamma$ as 
$\gamma= \gamma_0^m$, where $\gamma_0$ is a prime periodic orbit and $m \ge 1$. 
Then $\Lambda(\gamma) = l(\gamma_0)$. We have the following lemma.
 
 \begin{lemma} \label{sameabscissa}
 The series 
 \[
 \sum_{\substack{\gamma \in \mathcal P(\phi)' : \\ [\gamma]=0}}
 e^{-sl(\gamma)}
\quad \mathrm{and} \quad
\sum_{\substack{\gamma \in \mathcal P(\phi) : \\ [\gamma]=0}}
 \frac{\Lambda(\gamma)}{l(\gamma)} e^{-sl(\gamma)}
 \]
 each have abscissa of convergence equal to $h(Y)$.
 \end{lemma}
 
 \begin{proof}
First we note that 
\[
 \sum_{\substack{\gamma \in \mathcal P(\phi) : \\ [\gamma]=0}}
e^{-sl(\gamma)}
-  \sum_{\substack{\gamma \in \mathcal P(\phi)' : \\ [\gamma]=0}}
 e^{-sl(\gamma)}
= \sum_{m=2}^\infty \sum_{\substack{\gamma \in \mathcal P(\phi)' : \\ [\gamma^m]=0}}
 e^{-sml(\gamma)}.
 \]
 The abscissa of convergence of the Right Hand Side can be bounded by
 \begin{align*}
 \lim_{T \to \infty} \frac{1}{T} \log
 \sum_{m=2}^\infty \sum_{\substack{\gamma \in \mathcal P(\phi)' : \\ ml(\gamma) \le T,
 \ [\gamma^m]=0}}
 1
 &=
 \lim_{T \to \infty} \frac{1}{T} \log
 \sum_{m=2}^\infty \sum_{\substack{\gamma \in \mathcal P(\phi)' : \\ ml(\gamma) \le T,
 \ [\gamma^m]=0}}
 e^{\langle \xi,[\gamma^m]_{\mathrm{TF}}\rangle}
 \\
 &\le
  \lim_{T \to \infty} \frac{1}{T} \log
 \sum_{m=2}^\infty \sum_{\substack{\gamma \in \mathcal P(\phi)' : \\ ml(\gamma) \le T}}
 e^{\langle \xi,[\gamma^m]_{\mathrm{TF}}\rangle}
 \\
 &\le
  \lim_{T \to \infty} \frac{1}{T} \log
 \sum_{m=2}^{[T/l_0]} \sum_{\substack{\gamma \in \mathcal P(\phi) : \\ l(\gamma) \le T/2}}
 e^{\langle \xi,[\gamma^m]_{\mathrm{TF}}\rangle}
 \\
 &= P(\langle \xi,\Psi \rangle)/2 = h(Y)/2,
 \end{align*}
 where $l_0$ denotes the period of the shortest orbit in $\mathcal P(\phi)$.
 The second statement follows from a similar argument.

 \end{proof}

 As a part of our approach to Theorem \ref{intro-main} will capture information about an Anosov
flow $\phi^t : M \to M$ and its lifts in terms of symbolic dynamical systems: subshifts of finite type and their skew-product extensions. We will introduce these systems in the next section and then go on to
discuss their relation to Anosov flows in section \ref{symbolicdynamics}.

\section{Subshifts of Finite Type and Group Extensions} \label{section:ssftge}

In this section we will define countable state Markov shifts and discuss some of their properties.
Basic definitions and results are taken from chapter 7 of \cite{kitchens}.
Let us emphasise that throughout we have the hypothesis that our Markov shifts are locally compact (this excludes examples such as the infinite full shift and the renewal shift). We shall be particularly concerned with finite state shifts and
skew product extensions of these by
a countable group.

Let $S$ be a countable set, called the {\it alphabet}, and let $A$ be a matrix, 
called the transition matrix, indexed by 
$S \times S$ with entries zero
or one. We then define the space
\[
\Sigma^+ =\Sigma_A^+ = \left\{x = (x_n)_{n=0}^\infty \in S^{\mathbb Z^+} \hbox{ : }
A(x_n,x_{n+1})=1 \ \forall n \in \mathbb Z^+\right\},
\]
with the product topology induced by the discrete topology on $S$. This topology is compatible
with the metric $d(x,y)= 2^{-n(x,y)}$, where
\[
n(x,y) = \inf\{n \hbox{ : } x_n \neq y_n\},
\]
with $n(x,y)=\infty$ if $x=y$. 
If $S$ is finite then $\Sigma^+$ is compact.
We say that $A$ is {\it locally finite} if all its row and column sums are finite.
Then $\Sigma^+$ is locally compact if and only if $A$ is locally finite.
(The skew product extensions we consider have this latter property.)

We define the (one-sided) countable state topological Markov shift
$\sigma : \Sigma^+ \to \Sigma^+$ by
$(\sigma x)_n = x_{n+1}$.  This is a continuous map.
We will say that $\sigma$ is {\it topologically
transitive} if it has a dense orbit and {\it topologically mixing} if, given non-empty 
open sets $U,V \subset \Sigma^+$, there exists
$N \geq 0$ such that $\sigma^{-n}(U) \cap V \neq \varnothing$ for all $n \geq N$. 
We say that the matrix $A$ is {\it irreducible} if, for each $(i,j) \in S \times S$, there exists
$n =n(i,j)\geq 1$ such that $A^n(i,j)>0$.
For $A$ irreducible, set $p\geq 1$ to be the greatest common divisor of periods of periodic orbits
$\sigma : \Sigma^+ \to \Sigma^+$; this $p$ is called the period of $A$.
We say that $A$ is {\it aperiodic} if $p=1$ or, equivalently,
if there exists $n \geq 1$ such that $A^n$ has all entries positive.
Suppose that $A$ is locally finite. Then $\sigma : \Sigma^+ \to \Sigma^+$ is topologically
transitive if and only if $A$ is irreducible and $\sigma : \Sigma^+ \to \Sigma^+$ is topologically
mixing if and only if $A$ is aperiodic.

Suppose that $A$ is irreducible but not aperiodic and fix $i \in S$. Then we may partition
$S$ into sets $S_l$, $l=0,\ldots,p-1$, defined by
\[
S_l = \{j \hbox{ : } A^{np+l}(i,j)>0 \mbox{ for some } n \geq 1\}.
\]
(This partition is independent of the choice of $i$.)
For each $l$, let $A_l$ denote the restriction of $A$ to $S_l \times S_l$; then 
$\sigma : \Sigma_{A_l}^+ \to \Sigma_{A_{l+1}}^+$ (mod $p$) and $A_l^p$
is aperiodic.

We say that an $n$-tuple
$w = (w_0,\ldots,w_{n-1}) \in S^n$ is an {\it allowed word} of length $n$ if
$A(w_j,w_{j+1})=1$ for $j =0,\ldots,n-2$. 
We will write $\mathcal W^n$ for the set of allowed words of length $n$.
If $w \in \mathcal W^n$ then we define 
the associated length $n$ cylinder set $[w]$ by 
\[
[w] = \{x \in \Sigma_A^+ \hbox{ : } x_j =w_j, \ j=0,\ldots,n-1\}.
\]

For a function $f : \Sigma^+ \to \mathbb R$, set
\[
V_n(f) = \sup\{|f(x)-f(y)| \hbox{ : } x_j=y_j, \ j=0,\ldots,n-1\}.
\]
We say that $f$ is {\it locally H\"older continuous} if there exist $\alpha>0$ and $C \geq 0$ such that, for all $n \geq 1$,
$V_n(f) \leq C2^{-n\alpha}$.
(There is no requirement on $V_0(f)$ and a locally H\"older $f$ may be unbounded.)
The minimal possible $C$ is called the $\alpha$-H\"older seminorm

Suppose that $\sigma : \Sigma^+ \to \Sigma^+$ is topologically
transitive and let 
$f : \Sigma^+ \to \mathbb R$ be a locally H\"older continuous function. 
Following 
Sarig \cite{Sarig},
we define the
{\it Gurevi\v{c} pressure}, $P_{\mathrm{Gur}}(f,\sigma)$, of $f$ to be
\[
P_{\mathrm{Gur}}(f,\sigma) = \limsup_{n \to \infty} \frac{1}{n} \log 
\sum_{\substack{\sigma^n x=x \\ x_0=a}}
e^{f^n(x)},
\]
where $a \in S$. (The definition is independent of the choice of $a$.)

\begin{remark}
In \cite{Sarig}, Sarig gives this definition in the case where
 $\sigma : \Sigma^+ \to \Sigma^+$ is topologically mixing.
 However, the above decomposition of $\Sigma^+ = \Sigma_{A_0}^+ \cup \cdots
 \cup \Sigma_{A_{p-1}}^+$, with $\sigma^p$ topologically mixing on each component,
 together with the regularity of the function $f$,
 shows that the same definition may be made in the topologically transitive case.
 \end{remark}
 
 It is immediate from the definition that if $f \le f'$ then
 $P_{\mathrm{Gur}}(f,\sigma) \le P_{\mathrm{Gur}}(f',\sigma)$ and that for any constant $c \in \mathbb R$
 we have
  $P_{\mathrm{Gur}}(f+c,\sigma) \le P_{\mathrm{Gur}}(f,\sigma)+c$.
  We say that $f$ and $f'$ are {\it cohomologous} if their difference takes the form
  $f-f' = u\circ \sigma-u$. It is also clear from the definition that if $f$ and $f'$ are cohomologous then
   $P_{\mathrm{Gur}}(f,\sigma) = P_{\mathrm{Gur}}(f',\sigma)$.
 We note the following useful lemma.
 
 \begin{lemma} \label{strictlydecreasing}
 If $r: \Sigma^+ \to \mathbb R$ and $f : \Sigma^+ \to \mathbb R$ are locally H\"older 
 and $r$ is cohomolgous to a function $r'$ satisfying $\inf_{x \in \Sigma^+} r'(x) >0$ then
 $s \mapsto P_{\mathrm{Gur}}(-sr+f,\sigma)$ is strictly decreasing.
 \end{lemma} 
 
 \begin{proof}
 If $r$ and $r'$ are cohomologous then, for any $s \in \mathbb R$, $-sr$ and $-sr'$ are cohomologous.
 Let $c = \inf_{x \in \Sigma^+} r'(x) >0$ and suppose $t>s$. Then
 \[
 -tr' = -sr -(t-s)r' \le -sr -(t-s)c
 \]
 and therefore
 \begin{align*}
 P_{\mathrm{Gur}}(-tr+f,\sigma) 
 &= P_{\mathrm{Gur}}(-tr'+f,\sigma) \\
&\le P_{\mathrm{Gur}}(-sr' -(t-s)c+f,\sigma) \\
 &= P_{\mathrm{Gur}}(-sr'+f ,\sigma) -(t-s)c \\
 &<P_{\mathrm{Gur}}(-sr'+f ,\sigma)
 =P_{\mathrm{Gur}}(-sr+f ,\sigma).
 \end{align*}
 \end{proof}
 
We now specialise to the case where $S$ is finite. In this situation, we call 
$\sigma : \Sigma^+ \to \Sigma^+$
a (one-sided) subshift of finite type. The above definitions and results hold. If $f : \Sigma^+
\to \mathbb R$ is H\"older continuous then $f$ is locally H\"older. Provided
$\sigma : \Sigma^+ \to \Sigma^+$ is topologically transitive, the Gurevi\v{c} pressure
$P_{\mathrm{Gur}}(f,\sigma)$ agrees with the standard pressure $P(f,\sigma)$, defined by
\[
P(f,\sigma) = \limsup_{n \to \infty} \frac{1}{n} \log 
\sum_{\sigma^n x=x}
e^{f^n(x)}
\]
and if $\sigma$ is topologically mixing then the $\limsup$ may be replaced 
with a limit.

We now consider skew product extensions of a shift of finite type 
$\sigma : \Sigma^+ \to 
\Sigma^+$, which we will assume to be
topologically mixing. Let $G$ be a countable group (with identity element $e$) and let 
$\psi : \Sigma^+ \to G$ be a function
depending only on two co-ordinates, $\psi(x)=\psi(x_0,x_1)$. 

(One could consider more general $\psi$ but this set-up suffices for our application to Anosov flows.) 
This data defines a 
{\it group extension} or
{\it skew product
extension} $T_\psi : \Sigma^+ \times G \to \Sigma^+ \times G$ by
$T_\psi(x,g) = (\sigma x,g\psi(x))$. For $n \geq 1$ define
$\psi_n$ by
\[
\psi^n(x) =   \psi(x) \psi(\sigma x)\cdots \psi(\sigma^{n-1}x);
\]
then $T_\psi^n(x,g) = (x,g)$ if and only if
$\sigma^n x=x$ and $\psi_n(x)=e$.

The map $T_\psi : \Sigma^+ \times G \to \Sigma^+ \times G$ 
is itself a countable state Markov shift
with alphabet $S \times G$ and transition matrix $\widetilde A$ defined by
$\widetilde A((i,g),(j,h)) = 1$ if $A(i,j)=1$ and $\psi(i,j) = g^{-1}h$, and 
$\widetilde A((i,g),(j,h)) = 0$
otherwise.
Clearly, $\widetilde A$ is locally finite and so the topological transitivity and 
topological 
mixing of $\widetilde \sigma$ are equivalent to $\widetilde A$ being irreducible and 
aperiodic,
respectively.

Let $f : \Sigma^+ \to \mathbb R$ be H\"older continuous and define 
$\widetilde f : \Sigma_A^+ \times G \to \mathbb R$ by $\widetilde f(x,g) = f(x)$;
then $\widetilde f$ is locally H\"older continuous and its Gurevi\v{c} pressure 
$P_{\mathrm{Gur}}(\widetilde f,T_\psi)$ is defined. In fact,
it is easy to see that, due to the mixing of $\sigma$,
\[
P_{\mathrm{Gur}}(\widetilde f,T_\psi) = \limsup_{n \to \infty} \frac{1}{n} \log 
\sum_{\substack{\sigma^n x=x \\ \psi^n(x)=e}}
e^{f^n(x)}.
\]

 We end this section by discussing two-sided subshifts of finite type
and suspended flows over them. Given a finite alphabet $S$ and transition matrix
$A$, we define
\[
\Sigma= \Sigma_A = \left\{x = (x_n)_{n=0}^\infty \in S^{\mathbb Z} \hbox{ : }
A(x_n,x_{n+1})=1 \ \forall n \in \mathbb Z\right\}
\]
and the (two-sided) shift of finite type 
$\sigma : \Sigma \to \Sigma$ by
$(\sigma x)_n = x_{n+1}$. As before, we give $\Sigma$
with the product topology induced by the discrete topology on $S$ and this is compatible
with the metric $d(x,y)= 2^{-n(x,y)}$, where
\[
n(x,y) = \inf\{|n| \hbox{ : } x_n \neq y_n\},
\]
with $n(x,y)=\infty$ if $x=y$. 
Then $\Sigma$ is compact and $\sigma$ is a homeomorphism.
There is an obvious one-to-one correspondence between the periodic points of 
$\sigma : \Sigma \to \Sigma$ and $\sigma : \Sigma^+ \to \Sigma^+$.
Furthermore, we may pass from H\"older continuous functions on $\Sigma$ to H\"older
continuous functions on $\Sigma^+$ in such a way that sums around periodic orbits are preserved.
More precisely, we have the following lemma, due originally to Sinai \cite{Sinai72}, which appears as Proposition 1.2 of \cite{PP}.

\begin{lemma} \label{sinai}
Let $f : \Sigma \to \mathbb R$ be H\"older continuous. Then
there is a H\"older continuous function $f' : \Sigma^+ \to \mathbb R$ 
(with a smaller H\"older exponent) such that $f^n(x) =  (f')^n(x)$, whenever 
$\sigma^n x=x$.
\end{lemma}

We may also define suspended flows over $\sigma : \Sigma_A \to \Sigma_A$.
Given a strictly positive continuous function $r : \Sigma \to \mathbb R^+$, we define the 
$r$-suspension space 
\[
\Sigma^r = \{(x,s) \hbox{ : } x \in \Sigma, \ 0 \leq s \leq r(x)\}/\sim,
\]
where $(x,r(x)) \sim (\sigma x,0)$. The suspended flow $\sigma_r^t : \Sigma^r \to \Sigma^r$ is defined by
$\sigma_r^t(x,s) = (x,s+t)$ modulo the identifications. Clearly, there is a natural one-to-one
correspondence between periodic orbits for $\sigma_r^t : \Sigma^r \to \Sigma^r$ and periodic orbits for
$\sigma :  \Sigma \to \Sigma$.

Furthermore, if $\gamma$ is a periodic $\sigma_r$-orbit corresponding to the
periodic $\sigma$-orbit $\{x,\sigma x, \ldots, \sigma^{n-1}x\}$ then
the period of $\gamma$ is equal to $r^n(x)$.

\section{Gurevi\v{c} Entropy for Group Extensions}\label{section:gege}

In this section, we initiate the comparison between Gurevi\v{c} entropy and Gurevi\v{c} pressure  for group extensions of subshifts of finite type, and entropy and pressure for the base transformation and the abelianized extention. Carefully combining the result of Stadlbauer \cite{Stadlbauer13} and
the results of \cite{PollicottSharp94} will produce one direction of a proof of Theorem \ref{intro-gurevicentropy}. The other direction will lead us to prove a new result on Gurevi\v{c} pressure, which is appears in section \ref{amenabilityandpressure}.

Let $\sigma : \Sigma^+ \to \Sigma^+$ be a one-sided subshift of finite type.
For a countable group $G$ and a function $\psi : \Sigma \to G$, 

we consider the group extension
\[
T_\psi : \Sigma\times G \to \Sigma\times G
:
(x,g)\mapsto (\sigma x, g\psi(x)),
\]
which we assume to be transitive.

Let $f : \Sigma^+ \to \mathbb R$ be a H\"older continuous function.
We will also use $f$ to denote the function on $\Sigma^+ \times G$ defined by
$f(x,g)=f(x)$.
We wish to study the asymptotics of periodic points for $T_\psi$ (and compare them with those for
$\sigma$) when they are
weighted by $f$.
Of course (provided $G$ is infinite),
$T_\psi$ will have infinitely many periodic points with the same period but we will restrict 
to periodic points for which the second co-ordinate in the identity element.

In other words, we wish to compare the Gurevi\v{c} pressure
$P_{\mathrm{Gur}}(f,T_\psi)$ with the pressure $P(f,\sigma)$.
It is clear that 
$P_{\mathrm{Gur}}(f,T_\psi)\le P(f,\sigma)$ and it is natural to ask when equality holds.

This question has received considerable attention when the system and function exhibit a natural
``time-reversal'' symmetry.
Suppose there is a fixed point free involution $\kappa : S \to S$ such that
$A(\kappa j,\kappa i) = A(i,j)$, for all $i,j \in S$. We say that 
the skew product $T_\psi : \Sigma^+ \times G \to \Sigma^+ \times G$ is
{\it symmetric} (with respect to $\kappa$) if $\psi(\kappa j,\kappa i) = \psi(i,j)^{-1}$.
A function $f : \Sigma^+ \to \mathbb R$ is called {\it weakly symmetric} if, for all
$n \geq 1$ and and all length $n$ cylinders $[z_0,z_1,\ldots,z_{n-1}]$, there exists
$D_n >0$ such that $\lim_{n \to \infty} D_n^{1/n}=1$ and
\[
\sup_{\substack{x \in [z_0,\ldots,z_{n-1}] \\ y \in [\kappa z_{n-1},\ldots,\kappa z_{0}]}}
\exp(f^n(x)-f^n(y)) \leq D_n.
\]

The following is the main result of Stadlbauer \cite{Stadlbauer13},
restricted to the case where the base is a (finite state) subshift of finite
 type. We
 will use this in subsequent 
arguments.
(More generally, Stadlbauer considers skew product expansions of countable state
 Markov shifts.) We include his more general result on the spectral radius of the transfer operator (Theorem 5.4 \cite{Stadlbauer13}, there it is stated for pressure), the definition of which follows beneath the proposition.
 
 \begin{proposition} [Stadlbauer \cite{Stadlbauer13}, Theorem 5.4, and Theorems 4.1 and 5.6] \label{stadlbauer}
 Let $T_\psi : \Sigma^+ \times G \to \Sigma^+ \times G$ be a transitive
skew-product extension of a mixing subshift of finite type 
 $\sigma : \Sigma^+ \to \Sigma^+$ by a countable group $G$.
 
 If $G$ is non-amenable, then $\log \mathrm{spr}_{\mathcal H}(\mathcal L_f) < P(f,\sigma)$ for  $f : \Sigma_A^+ \to \mathbb R$ H\"older continuous.
 
 If, in addition, $T_\psi : \Sigma^+ \times G \to \Sigma^+ \times G$ is assumed to be symmetric and $f : \Sigma_A^+ \to \mathbb R$
 is a weakly symmetric H\"older continuous function, then we have $P_{\mathrm{Gur}}( f,T_\psi) = P(f,\sigma)$
 if and only if $G$ is amenable.
 
 \end{proposition}
 
 \begin{remark}
(i) In \cite{Stadlbauer13}, Stadlbauer considers skew products with $\psi$ depending on only one co-ordinate. However, replacing $S$ by $\mathcal W^2$, one can easily recover the above formulation.

\noindent
(ii) Setting $f=0$ immediately gives Theorem \ref{intro-gurevicentropy} for a symmetric extension.
 \end{remark}

In the absence of this symmetry, the answer becomes less clear.
In order to elaborate on this problem, we introduce 
the \emph{transfer operator} $\mathcal{L}_f$, defined pointwise by 
\[
\mathcal{L}_f v(x,g) = \sum_{\substack{y\in \Sigma: \\ \sigma y = x}} 
e^{f(y)}
v(y,g\psi(x)),
\]
for $v:\Sigma\times G \to \mathbb{C}$.  
In order to make use of the spectral properties of this operator, we restrict it to
the Banach space $\mathcal{H}$ of continuous functions $v$ for which $g\mapsto \|v(\cdot, g) \|_\infty$ is in 
$\ell^2(G)$,
with norm
\[
\|v\|_{\mathcal H} = \left(\sum_{g \in G} \|v(\cdot,g)\|_\infty^2\right)^{1/2}.
\]
If $\psi$ and $f$ satisfy the above symmetry conditions 
then
$P_{\mathrm{Gur}}(f,T_\psi)$ is equal to the logarithm of 
$\mathrm{spr}_{\mathcal{H}}(\mathcal{L}_f)$, the spectral radius of $\mathcal L_f : \mathcal H
\to \mathcal H$ \cite{Stadlbauer13}.
However, this equality does not hold in general and without symmetry we may have
$P_{\mathrm{Gur}}(f,T_\psi)< \log\mathrm{spr}_{\mathcal{H}}(\mathcal{L}_f)$.
On the other hand, 
there is a result of Jaerisch \cite{Jaerisch} that 
$\log \mathrm{spr}_{\mathcal{H}}(\mathcal{L}_f) = P(f,\sigma)$ if and only if $G$ is amenable,
and so it is clear that, when considering $P_{\mathrm{Gur}}(f,T_\psi)$, the pressure $P(f,\sigma)$ does not provide a useful comparison.
We shall show that, in fact, the natural comparison is between
$P_{\mathrm{Gur}}(f,T_\psi)$ and $P_{\mathrm{Gur}}(f,T_{\psi^{\mathrm{ab}}})$,
where $T_{\psi^{\mathrm{ab}}} : \Sigma \times G^{\mathrm{ab}} \to \Sigma \times G^{\mathrm{ab}}$
is the induced $G^{\mathrm{ab}}$-extension, where $G^{\mathrm{ab}} = G/[G,G]$ is the
abelianization of $G$. More precisely, if $\pi : G \to G^{\mathrm{ab}}$ is the natural projection then
$\psi^{\mathrm{ab}} = \pi \circ \psi$.

We now address the proof of Theorem \ref{intro-gurevicentropy}. One implication in the theorem is given by the next proposition. The other implication will follow from the more general result proved in the next section (Theorem \ref{intro-main2}).

\begin{proposition}
If $G$ is not amenable then $h_{\mathrm{Gur}}(T_\psi) <h_{\mathrm{Gur}}(T_{\psi^{\mathrm{ab}}})$.
\end{proposition}

\begin{proof}
We have $G^{\mathrm{ab}} = \mathbb Z^a \times G_0$, for some $a \ge 0$, where $G_0$ is a 
finite abelian group.
First suppose that $a>0$. The system $T_{\psi^{\mathrm{ab}}} : \Sigma^+ \times G^{\mathrm{ab}}
\to \Sigma^+ \times G^{\mathrm{ab}}$ induces a system on $\Sigma^+ \times \mathbb Z^a$, which we will still 
denote by $T_{\psi^{\mathrm{ab}}}$ and which has the same Gurevi\v{c} entropy.
Following the analysis of \cite{PollicottSharp94},
$h_{\mathrm{Gur}}(T_{\psi^{\mathrm{ab}}}) = P((\langle \xi,\psi^{\mathrm{ab}}\rangle,\sigma)$, for
some $\xi \in \mathbb R^a$.
Since $G$ is not amenable, the first part of Proposition \ref{stadlbauer} tells us that
\[
\log \mathrm{spr}_{\mathcal H}(\mathcal L_{\langle \xi,\psi^{\mathrm{ab}}\rangle}) <
P(\langle \xi,\psi^{\mathrm{ab}}\rangle,\sigma) =h_{\mathrm{Gur}}(T_{\psi^{\mathrm{ab}}}).
\]
and that, for any H\"older continuous $f : \Sigma^+ \to \mathbb R$,
$P_{\mathrm{Gur}}(f,T_\psi)) \le \log \mathrm{spr}_{\mathcal H}(\mathcal L_f)$.
Combining these statements gives that
\[
P_{\mathrm{Gur}}( \langle \xi,\psi^{\mathrm{ab}}\rangle ,T_\psi)) <h_{\mathrm{Gur}}(T_{\psi^{\mathrm{ab}}}).
\]
However, since $\psi^n(x)=e$ implies that $(\psi^{\mathrm{ab}})^n=0$,
\begin{align*}
P_{\mathrm{Gur}}( \langle \xi,\psi^{\mathrm{ab}}\rangle ,T_\psi))
&=
\limsup_{n \to \infty} \frac{1}{n} \log 
\sum_{\substack{\sigma^n x=x \\ \psi^n(x)=e}}
e^{\langle \xi,(\psi^{\mathrm{ab}})^n(x)\rangle}
\\
&= \limsup_{n \to \infty} \frac{1}{n} \log
\#\{x \in \Sigma^+ \hbox{ : } \sigma^n x=x, \ (\psi^{\mathrm{ab}})^n=0\}
=h_{\mathrm{Gur}}(T_\psi),
\end{align*}
giving the required strict inequality.

Now suppose that $a=0$, so that $G^{\mathrm{ab}}$ is finite. Then
$h_{\mathrm{Gur}}(T_{\psi^{\mathrm{ab}}}) = h(\sigma)$. As above, we have
\[
h_{\mathrm{Gur}}(T_\psi) \leq \log \mathrm{spr}_{\mathcal H}(\mathcal L_0)
< P(0,\sigma) = h(\sigma),
\]
completing the proof.
\end{proof}

\section{Gurevi\v{c} Pressure for Amenable Extensions} \label{amenabilityandpressure}

The purpose of this section is to prove the following result. Setting $f=0$ will complete the proof of Theorem
\ref{intro-gurevicentropy}. To simplify notation, for a function $f : \Sigma^+ \to \mathbb R$,
we will use $f$ to denote the induced functions on the group extensions $\Sigma^+ \times G$ and
$\Sigma^+ \times G^{\mathrm{ab}}$ (i.e. $f(x,g) =f(x)$ for all group elements $g$).

\begin{theorem}\label{intro-main2}
Assume that $T_\psi$ is transitive.
Let $f : \Sigma^+ \to \mathbb R$ be H\"older continuous.
If $G$ is amenable then
\[
P_{\mathrm{Gur}}(f,T_\psi) = P_{\mathrm{Gur}}(f,T_{\psi^{\mathrm{ab}}}) .
\]
\end{theorem}

The proof is inspired by Roblin's proof in \cite{Roblin} that if $\Gamma$ is a convex co-compact group 
of isometries of a $\mathrm{CAT}(-1)$ space and $\Gamma'$ is a normal subgroup such that
$\Gamma/\Gamma'$ is amenable then the critical exponents of $\Gamma$ and $\Gamma'$ are equal.
We will make use of a family of $\sigma$-finite measures $\nu_{\eta,g}$, indexed by $\Sigma^+ \times G$,
introduced by Stadlbauer \cite{Stadlbauer18}. For $t >0$, write
\[
\mathcal{P}(t) = \sum_{n\in\mathbb{N}} t^{-n} b_n \sum_{\substack{y\in \Sigma : \\ \sigma^n y=o}} e^{f^n(y)} \mathds{1}_{\Sigma^+ \times \{e\}}(T_\psi^n((y,e))
\]
for a chosen distinguished $o\in\Sigma^+$, and
where 
$b_n$ is a slowly diverging sequence chosen so that $\mathcal{P}(t)$ diverges at its 
radius of convergence. 
More precisely, if the terms $b_n$ are omitted then it is clear that the resulting series
converges for $t > e^{P_{\mathrm{Gur}}(f, T_\psi)}$ and diverges for $t < e^{P_{\mathrm{Gur}}(f, T_\psi)}$.
It is then possible to choose a non-decreasing sequence $b_n \ge1$ such that 
$\lim_{n \to \infty} b_n/b_{n+1} =1$, $\mathcal P(t)$ has radius of convergence
$e^{P_{\mathrm{Gur}}(f, T_\psi)}$ and diverges at $t= e^{P_{\mathrm{Gur}}(f, T_\psi)}$
(Lemma 3.1 of \cite{DenkerUrbanski}). 
For the rest of the section, we shall write $\rho:= e^{P_{\mathrm{Gur}}(f, T_\psi)}$.

For $t> \rho$, $\eta \in \Sigma$ 
and $g \in G$, define a measure $\nu^t_{\xi,g}$ on $\Sigma^+ \times G$ by the formula
\[
\nu^t_{\eta,g}(v) := 
\int_{\Sigma^+ \times G} v \, d\nu^t_{\eta,g}
= \frac{1}{\mathcal{P}(t)}\sum_{n\in\mathbb{N}} t^{-n} b_n
\sum_{\substack{z\in \Sigma^+ \times G: T_\psi^n(z)=(\eta, g)}} e^{f^n(z)}v(z),
\]
for each continuous function $v:\Sigma^+ \times G\to \mathbb{R}$.
Now let $t \to \rho+$, and choose a weak limit $\nu_{\eta,g}$. 
We can do this because each $\Sigma^+ \times \{g\}$ is compact, and 
$G$ is countable. 
We can also ensure that there is a countable dense subset of $\eta$ for which the limit is 
attained along the same subsequence. In Theorem 5.1 of \cite{Stadlbauer18}, it is shown how to extend this to all 
$\eta \in \Sigma^+$ using H\"{o}lder continuity. For the proof of our Theorem \ref{intro-main2}, we will only use the countable collection of points $\eta = wo$, for 
$w\in \bigcup_{n\in\mathbb{N}}\mathcal{W}^n_{o}$, 
where
$o$ is the chosen distinguished element of $\Sigma^+$
and $\mathcal W_o^n$ denotes the set of elements $w$ of $\mathcal W^n$ for which $wo \in \Sigma^+$.

The following constants will frequently appear:
\[
B_f = \inf_{z \in \Sigma^+} e^{f(z)} 
\quad \mathrm{and}
\quad
C_f = \exp \left(\frac{|f|_\alpha}{1-2^{-\alpha}}\right),
\]
where $\alpha>0$ is the H\"older exponent of $f$ and $|f|_\alpha$ is the $\alpha$-H\"older
seminorm of $f$.

\begin{lemma}\label{easyinequality}

\noindent 
(i) There exists $C>0$ such that, for any non-negative continuous function $v : \Sigma^+ \times G \to \mathbb R$ and any $(\eta,g),(\xi,h) \in \Sigma^+ \times G$ satisfying
$T^k_\psi(\eta,g)=(\xi,h)$, we have
\[
\nu_{\eta,g}(v)\ge C^k\nu_{\xi,h}(v).
\]
\item

\noindent
(ii)
For any non-negative continuous function $v : \Sigma^+ \times G \to \mathbb R$ 
and any $\eta,\xi \in \Sigma^+$ belonging to the same cylinder of length $1$, we have
\[
\nu_{\eta,g}(v)\ge C_f^{-1} \nu_{\xi,g}(v).
\]
\end{lemma}

\begin{proof}
Let $v : \Sigma^+ \times G \to \mathbb R$ be an indicator function on some cylinder. It will be sufficient to prove the lemma for functions of this form, the general non-negative, continuous case follows by approximating by linear combinations of indicator functions.

We proceed with part (i). Let $(\eta,g),(\xi,h) \in \Sigma^+ \times G$ with $T^k_\psi(\eta,g)=(\xi,h)$. 
For $t > \rho$ we have,
\begin{align*}
&\nu^t_{\eta,g}(v) 
= 
\frac{1}{\mathcal{P}(t)}\sum_{n=1}^\infty t^{-n}b_n\sum_{w\in \mathcal{W}_\eta^n } 
e^{f^n(w\eta)} v(w\eta, g\psi^n(w)) \\
&\ge
\frac{1}{\mathcal{P}(t)}\sum_{n= k+1}^\infty t^{-k}t^{-(n-k)}\frac{b_n}{b_{n-k}}b_{n-k}
\sum_{u\in \mathcal{W}_\xi^{n-k}   } 
e^{f^{k}(\xi)}e^{f^{n-k}(u\xi)} v(u\xi, h\psi^{n-k}(u)) \\
&\ge
\left(\sup_{n \in \mathbb{N}} \frac{b_n}{b_{n-1}}\right)^k
B_f^k t^{-k} \frac{1}{\mathcal{P}(t)}\sum_{m=1}^\infty t^{-{m}}b_m
\sum_{u\in \mathcal{W}_\xi^{m}} e^{f^{m}(u\xi)} v(u\xi, h\psi^{m}(u)).
\end{align*}
Taking weak limits 
as $t \to \rho+$ gives the conclusion (with $C=B_f(\sup_{n\in\mathbb{N}} b_n/b_{n-1})$).

Now, for part (ii) we assume that $\eta,\xi$ belong to the same cylinder. Recall that $v$ is assumed to be an indicator function, we write $v=\mathds{1}_{[u]}$ for some $u$, and let $k$ be the length of $u$.
For $t > \rho$ we have,
\begin{align*}
\nu^t_{\eta,g}(v) &= \frac{1}{\mathcal{P}(t)}\sum_{n=1}^\infty t^{-n}b_n
\sum_{w\in \mathcal{W}_\eta^n} e^{f^n(w\eta)} v(w\eta, g\psi^n(w)) \\
&=
\frac{1}{\mathcal{P}(t)}\sum_{n=1}^\infty t^{-n}b_n
\sum_{w\in \mathcal{W}_\xi^n} e^{f^n(w\eta)-f^n(w\xi)}
e^{f^n(w\xi)} v(w\eta, g\psi^n(w)) \\
&\ge
\frac{1}{\mathcal{P}(t)}\sum_{n= k}^\infty t^{-n}b_n
\sum_{w\in \mathcal{W}_\xi^n} C_f^{-1} e^{f^n(w\xi)} v(w\xi, g\psi^n(w)) \\
&=
C_f^{-1} \frac{1}{\mathcal{P}(t)}\sum_{m=1}^\infty t^{-m}b_m
\sum_{w\in \mathcal{W}_\xi^m}  e^{f^m(w\xi)} v(w\xi, g\psi^m(w)) \\
&-
\frac{1}{\mathcal{P}(t)}\sum_{i=1}^k t^{-i}b_i
\sum_{w\in \mathcal{W}_\xi^i} C_f^{-1} e^{f^i(w\xi)} v(w\xi, g\psi^i(w)) \\
\end{align*}
Since $\mathcal{P}(t)\to \infty$, as $t \to \rho+$, 
it follows that,
\[
\nu^t_{\eta,g}(v) \ge C_f^{-1} \nu^t_{\xi,g}(v)
\]
as required.
\end{proof}

\begin{lemma} \label{itsbounded}
Let $v$ be a non-negative continuous function which is strictly positive on $\Sigma^+\times\left\{ e \right\}$.
If $T_\psi$ is transitive then $\nu_{z,g}(v)>0$. Furthermore, for each $a \in G$, we have
\[
\sup_{g\in G} \frac{\nu_{z,ga}(v)}{\nu_{z,g}(v)}<\infty.
\]

\end{lemma}

\begin{proof}
Let $g,a\in G$ be arbitrary. We write $z_0$ for the first letter of $z \in \Sigma^+$.

If $T_\psi$ is transitive then there are $\eta\in [z_0]$, $\xi\in [o]$ and $k_1,k_2 \ge 0$ such that 
$T^{k_1}_\psi(\eta,g) = (z,ga)$ and $T^{k_2}_\psi(z,ga) = (\xi,e)$. 
Then, by Lemma \ref{easyinequality}, we have that 
\[
\nu_{\eta,g}(v) \ge C^{k_1}\nu_{z,ga}(v)\ge C^{k_1}C^{k_2}\nu_{\xi,e}(v).
\]
Since $v$ is bounded from below away from zero on $\Sigma^+\times\left\{ e \right\}$, there is $c>0$ with
$\nu_{\xi,e}(v)\ge c\nu_{\xi,e}(\Sigma^+\times\left\{ e \right\})$. Then, $\nu_{\xi,e}(\Sigma^+\times\left\{ e \right\})\ge C_f^{-1} \nu_{o,e}(\Sigma^+\times\left\{ e \right\})$, and by construction $\nu_{o,e}(\Sigma^+\times\left\{ e \right\})=1$.
This gives the conclusion that $\nu_{z,g}(v)>0$ for all $g\in G$.

To obtain the final statement, we note that
\[
\frac{\nu_{z,ga}(v)}{\nu_{\eta,g}(v)}\le C^{-k_1}
\]
Moreover, since $\eta,z$ are in the same cylinder, the second part of Lemma \ref{easyinequality} gives that
\[
\frac{\nu_{z,ga}(v)}{\nu_{z,g}(v)} \le C^{-k_1}C_f.
\]
\end{proof}

\begin{lemma}\label{almosteig}
For any continuous function $v : \Sigma^+ \times G \to \mathbb R$ and any $g,a \in G$, we have
that
\[
\rho \nu_{\eta,g}(v) =  \sum_{u\in\mathcal{W}_\eta^1}  e^{f(u\eta)}\nu_{u\eta,g\psi(u)}(v).
\]
\end{lemma}

\begin{proof}
For $t > \rho$, we have
\begin{align*}
\nu_{\eta,g}^t(v) 
&=  \frac{1}{\mathcal{P}(t)}\sum_{n=1}^\infty t^{-n}b_n
\sum_{w\in \mathcal{W}_\eta^n} e^{f^n(z)}v{(w\eta, g\psi^n(w))} \\
&=
\frac{1}{\mathcal{P}(t)}\sum_{u\in\mathcal{W}_\eta^1}
\sum_{n=1}^\infty t^{-n}b_n\sum_{w\in \mathcal{W}_{u\eta}^{n-1}} 
e^{f^n(wu\eta)} v{(wu\eta, g\psi(u)\psi^{n-1}(w))} 
\\
&=
\frac{1}{\mathcal{P}(t)}\sum_{u\in\mathcal{W}_\eta^1}  t^{-1} e^{f(u\eta)} 
\sum_{n=1}^\infty \frac{b_n}{b_{n-1}}t^{-(n-1)}b_{n-1}
\\
&\times
\sum_{w\in \mathcal{W}_{u\eta}^{n-1}} e^{f^{n-1}(wu\eta)} v{(wu\eta, g\psi(u)\psi^{n-1}(w))}.
\end{align*}
Letting $t \to \rho+$, we obtain
\[
\nu_{\eta,g}(v) = \rho^{-1} \sum_{u\in\mathcal{W}_\eta^1}  e^{f(u\xi)}\nu_{u\eta,g\psi(u)}( v),
\]
where we have used the divergence of $\mathcal{P}(t)$ as $t\to \rho+$ and that $\lim_{n\to\infty} b_n/b_{n-1}=1$.
\end{proof}

To complete the proof of
Theorem \ref{intro-main2},
we introduce the trick of Roblin -- when $G$ is amenable we can almost project $\nu$ to an eigenfunction 
for $\Sigma$.

\begin{proof}[Proof of Theorem \ref{intro-main2}]
Let $\mathfrak M$ be the Banach mean for $G$. We fix a basepoint $o\in \Sigma$.
Let $v : \Sigma^+ \times G$ be a non-nefative continuous function.
Since $T_\psi$ is transitive, iterating Lemma \ref{almosteig} gives that, for some $C$ independent of $n$,
\[
\rho^n \nu_{o,g}(v) =  \sum_{u\in\mathcal{W}^n_o} e^{f^n(uo)}\nu_{uo,g\psi^n(u)}(v)\ge C\sum_{u\in\mathcal{W}^n_o} e^{f^n(uo)}\nu_{o,g\psi^n(u)}(v)
\]
(where we use the Lemma \ref{easyinequality} to compare $\nu_{uo,g\psi^n(u)}(v)$ with 
$\nu_{o,g\psi^n(u)}(v)$).

We cannot apply the mean directly in the above inequality
as $\nu_{o,g\psi^n(u)}(v)$ may not be bounded in $g$.
Instead, Lemma \ref{itsbounded} tells us that we may normalise by $\nu_{o,g}(v)$ to obtain a bounded 
function on $G$. 
Recall Jensen's inequality: let $(X,\mu)$ be a measure space and $h\in L^1(\mu)$. If $\phi:\mathbb{R}\to \mathbb{R}$ is convex, then
\[
\phi\left(\int h \, d\mu \right)\le \int \phi\circ h \, d\mu
\]
The inequality is also true for the Banach mean $\mathfrak M$ in place of the countably additive 
$\int \cdot \, d\mu$, since the proof only uses monotonicity of the integral and linearity. 
The reverse inequality is given when $\phi$ is concave. 
We apply this with concave function $\phi = \log$ to obtain
\begin{align*}
C^{-1} \rho^n 
\mathfrak M\left[ g\mapsto \frac{\nu_{o,g}(v)}{\nu_{o,g}(v)} \right] &\ge  
\mathfrak M\left[ g\mapsto \sum_{u\in\mathcal{W}^n_o} 
e^{f^n(uo)}\frac{\nu_{o,g\psi^n(u)}(v)}{\nu_{o,g}(v)} \right]  \\
&= \sum_{u\in\mathcal{W}^n_o} e^{f^n(uo)} 
\mathfrak M\left[ g\mapsto \frac{\nu_{o,g\psi^n(u)}(v)}{\nu_{o,g}(v)} \right]  \\
&\ge \sum_{u\in\mathcal{W}^n_o} e^{f^n(uo)} 
\exp \mathfrak M\left[ g\mapsto \log \frac{\nu_{o,g\psi^n(u)}(v)}{\nu_{o,g}(v)} \right].
\end{align*}
The last function is important to us and we give it a name: we define
$\chi : G \to \mathbb R$ by
\[
\chi(a)={\mathfrak M\left[g\mapsto \log \frac{\nu_{o,ga}(v)}{\nu_{o,g}(v)} \right]}.
\]
We claim that $\chi$ is a homomorphism.
To see this, we first check that $\chi(ab) = \chi(a)+\chi(b)$. Firstly,
\[
\frac{\nu_{o,gba}(v)}{\nu_{o,g}(v)} = \frac{\nu_{o,gba}(v)}{\nu_{o,gb}(v)}\frac{\nu_{o,gb}(v)}{\nu_{o,g}(v)} ,
\]
and when we take the mean, by right invariance, we have 
\[
\mathfrak M\left[g\mapsto \log \frac{\nu_{o,gba}(v)}{\nu_{o,gb}(v)}\right] =
\mathfrak M\left[g\mapsto \log \frac{\nu_{o,ga}(v)}{\nu_{o,g}(v)}\right].
\]
We also have to check that $\chi(e)=0$, but this is immediate.

We can now bound the Gurevi\v{c} pressure.
Since any homomorphism factors through the abelianisation of $G$, if 
$(x,0) \in \Sigma^+ \times G^{\mathrm{ab}}$ satisfies $T_{\psi^{\mathrm{ab}}}^n(x,0)=(x,0)$
then $\chi(\psi^n(x))=0$.
For $u \in \mathcal W^n$, we write $u^\infty$ for the infinite 
concatenation of copies of $u$. 
Since the transition matrix $A$ is aperiodic, there exists an $N \ge 1$ such that, for every $n \ge 1$ and every 
$u \in \mathcal W^n$ such that $u^\infty \in \Sigma^+$, there exists an admissible word $u^\#$ of length $N+1$ such that $(u^\#)_0 =u_0$ and 
$uu^\# \in \mathcal W_o^{n+N+1}$. Since $u^\infty$ and $uu^\# o$ agree in the first $n+1$ places, we have
$\psi^n(u^\infty) = \psi^n(uu^\# o)$ and
\[
e^{f^n(u^\infty)} \le C_f e^{f^n(uu^\# o)}.
\]
Thus we have
\begin{align*}
\sum_{\substack{x\in \Sigma^+: \; \sigma^nx=x \\ T^n_{\psi^{\mathrm{ab}}}(x,0)=(x,0)}} e^{f^n(x)} 
&=
\sum_{\substack{x\in \Sigma^+: \: \sigma^nx=x \\ T^n_{\psi^{\mathrm{ab}}}(x,0)=(x,0)}} 
e^{\chi(\psi^n(x))}e^{f^n(x)} \\
&=
\sum_{\substack{u \in \mathcal W^n : \: u^\infty \in \Sigma^+ \\ (\psi^{\mathrm{ab}})^n(u^\infty)=0}}
e^{\chi(\psi^n(u^\infty))} e^{f^n(u^\infty)} \\
&\le
C_f \sum_{\substack{u \in \mathcal W^n : \: u^\infty \in \Sigma^+ \\ (\psi^{\mathrm{ab}})^n(u^\infty)=0}}
e^{\chi(\psi^n(u u^\# o)))} e^{f^n(u u^\# o)} \\
&\le
C_f C_0^{-(N+1)} D^{-(N+1)}\sum_{w \in \mathcal W_o^{n+N+1}} 
e^{\chi(\psi^{n+N+1}(wo))} e^{f^{n+N+1}(wo)}  
\\
&\le C^{-1} C_f C_0^{-(N+1)} D^{-(N+1)} \rho^{n+N+1},
\end{align*}
where
$D = \inf_{x \in \Sigma^+} e^{\chi(\psi(x))}$.
We conclude that $e^{P_{\mathrm{Gur}}(f, T_{\psi^{\mathrm{ab}}})}\le \rho =e^{P_{\mathrm{Gur}}(f, T_\psi)}$. The other inequality is trivially true, therefore 
$P_{\mathrm{Gur}}(f, T_{\psi^{\mathrm{ab}}})= P_{\mathrm{Gur}}(f, T_\psi)$.
\end{proof}

\section{Symbolic Dynamics for Anosov Flows}\label{symbolicdynamics}

We begin by discussing the symbolic coding of Anosov
flows (and, more generally, hyperbolic flows) introduced by Ratner \cite{Ratner} and Bowen \cite{Bowen}.
Let $M$ be a smooth compact Riemannian manifold and let $\phi^t : M \to M$ be a $C^1$ flow. 
A closed, $\phi^t$-invariant set $\Lambda\subset M$ is said to be \emph{hyperbolic} if 
there is a continuous, $D\phi^t$-invariant splitting of the tangent bundle
\[
T_{\Lambda}(M) = E^0\oplus E^s \oplus E^u
\]
and constants $C,\lambda>0$ such that $E^0$ the line bundle tangent to the flow direction and
\begin{enumerate}
\item $\|D\phi^t v\| \le C e^{-\lambda t}\|v\|$ for all $v\in E^s$;
\item $\|D\phi^{-t} v\| \le C e^{-\lambda t}\|v\|$ for all $v\in E^u$.
\end{enumerate}
(We remark that this definition is independent of the choice of metric when $\Lambda$ is compact.)
If $M$ is a hyperbolic set then we call $\phi^t : M \to M$ an Anosov flow.
There exist examples of Anosov flows which are not transitive
\cite{FranksWilliams}, \cite{BBGR-H} but we shall always assume that transitivity holds, so we have that
$M = \Omega(\phi)$, the non-wandering set for $\phi$. By a fundamental result of Anosov
\cite{Anosov}, the set of periodic orbits of $\phi$ is dense in $\Omega(\phi)$ and hence, in our setting, in $M$.

We now describe some of the constructions which play an important role in the symbolic coding of transitive Anosov flows.
For $x\in  M$ define the (strong) \emph{local stable manifold} $W_\epsilon^s(x)$ and (strong) \emph{local unstable manifold} $W_\epsilon^u(x)$ by
\[
W_\epsilon^s(x) = \left\{y\in M : d(\phi^t(x),\phi^t(y))\le \epsilon \text{ for all }t, \lim_{t\to \infty}d(\phi^t(x),\phi^t(y))=0 \right\},
\]
\[
W_\epsilon^u(x) = \left\{y\in M : d(\phi^{-t}(x),\phi^{-t}(y))\le \epsilon \text{ for all }t, \lim_{t\to \infty}d(\phi^{-t}(x),\phi^{-t}(y))=0 \right\}.
\]
For small enough $\epsilon>0$, these sets are diffeomorphic to embedded disks of 
dimension $d^s$ and $d^u$, respectively, where $d^s + d^u = \dim M -1$. These sets give us a \emph{local product structure} $[\cdot,\cdot]$. For sufficiently close $x,y$, we have that $W_\epsilon^s(x)\cap W^u_\epsilon(\phi^t(y))\ne\varnothing$ for a unique $t\in[-\epsilon,\epsilon]$, and we define $[x,y]$ to be this intersection point.

Suppose that $D_1,\ldots , D_k$ are codimension $1$ disks that form a local cross-section to the flow.
We say that $R_i\subset \mathrm{int}(D_i)$ is a \emph{rectangle} if $x,y\in R_i$ implies that 
$[x,y]= \phi^t z$, for some $z\in D_i$, $t\in [-\epsilon,\epsilon]$,
where the interior is taken relative to $D_i$.
We say that $R_i$ is \emph{proper} if $\overline{\mathrm{int}(R_i)}=R_i$, where again the interior is 
taken relative to $D_i$.

Write $P$ for the Poincar\'{e} map $P: \bigcup_{i=1}^kR_i\to \bigcup_{i=1}^kR_i$. Write $W^s_\epsilon(x,R_i)$ and $W^u_\epsilon(x,R_i)$ for the projection of $W^s_\epsilon(x)$ and $W^u_\epsilon(x)$ onto $R_i$ respectively. We say that $\mathcal{R}=\left\{R_1,\cdots, R_k\right\}$ is a \emph{Markov section} if
\begin{enumerate}
\item $x\in \mathrm{int}(R_i)$ and $Px\in \mathrm{int}(R_j)$ implies that $P(W^s_\epsilon(x,R_i))\subset W^s_\epsilon(Px,R_j))$; and 
\item $x\in \mathrm{int}(R_i)$ and $P^{-1}x\in \mathrm{int}(R_j)$ implies that $P^{-1}(W^u_\epsilon(x,R_i))\subset W^s_\epsilon(P^{-1}x,R_j))$.
\end{enumerate}

\begin{proposition}[Bowen \cite{Bowen}, Ratner \cite{Ratner}]
For all sufficiently small $\epsilon>0$, $\phi^t$ has a Markov section $
\mathcal{R}=\left\{R_1,\cdots, R_k\right\}$ such that $\mathrm{diam}(R_i)\le \epsilon$ for each $i$, and $\bigcup_{t\in [-\epsilon,\epsilon]}\phi^t(\cup_{i=1}^k R_i)= M$.
\end{proposition}

These Markov sections provide us with a ``symbolic coding'' for the geodesic flow. In the following, the Markov section $\mathcal{R}=\left\{R_1,\cdots, R_k\right\}$ plays the role of an alphabet for a subshift of finite type $\Sigma$ with transition matrix $A$, defined by $A(i,j)=1$ if there is $x\in \mathrm{int}(R_i)$ with $Px\in \mathrm{int}(R_j)$. 

\begin{proposition}[Bowen \cite{Bowen}, Bowen and Ruelle \cite{BowenRuelle}]\label{Bowen}
There is a mixing subshift of finite type $\sigma :\Sigma\to \Sigma$ and a strictly positive H\"{o}lder 
continuous potential $r:\Sigma\to\mathbb{R}^+$ such that the suspended flow $\sigma^t_r: \Sigma^r\to \Sigma^r$ is semi-conjugate to  $\phi^t:M\to M$. More precisely, there exists a 
bounded-to-one surjective H\"{o}lder continuous
function  $\theta: \Sigma^r\to M$ such that $\theta\circ \sigma_r^t = \phi^t\circ \theta$. 
Furthermore, if $f : M \to \mathbb R$ is H\"older continuous  then
$f \circ \theta : \Sigma^r \to \mathbb R$ is H\"older continuous
and $\theta$ is a measure theoretic isomorphism between the equilibrium states of $f \circ \theta$ and $f$.
In particular, $P(f\circ \theta) = P(f)$.
\end{proposition}

The semi-conjugacy $\theta : \Sigma^r \to M$ is not in general a bijection, so results on counting orbits 
do not immediately translate between settings. 
In particular, there is overcounting of the periodic orbits that pass through the boundaries of the sections.
However, this discrepancy may be accounted for by the following result of Bowen (extending work of Manning \cite{Manning} in the diffeomorphism case).

\begin{lemma}[Bowen \cite{Bowen}]\label{remainders}
There are finitely many additional subshifts of finite type $\sigma_i :\Sigma_i \to \Sigma_i$, $i=1,\ldots,q$,
with corresponding 
strictly positive H\"older continuous functions $r_i : \Sigma_i \to \mathbb R$
and H\"older continuous maps $\theta_i : \Sigma_i^{r_i} \to M$, which are bounded-to-one but not surjective, such that
$\theta_i\circ \sigma_r^t = \phi^t\circ \theta_i$, $i=1,\ldots,q$, and such that, if $\nu(\phi, T)$ 
(respectively, $\nu(\sigma_r,T)$, $\nu(\sigma_{r_i},T)$) denotes the number of periodic 
$\phi$-orbits (respectively, $\sigma_r$-orbits, $\sigma_{r_i}$-orbits) with period equal to $T$ then
\[
\nu(\phi,T) = \nu(\sigma_r,T) + \sum_{i=1}^q \epsilon_i \nu(\sigma_{r_i},T),
\]
with $\epsilon_i\in \{ -1,1\}$.
 \end{lemma}

Since the $\theta_i$ are {\it not} surjective, for a H\"older continuous function $F : M \to \mathbb R$
we have $P(F \circ \theta_i, \sigma_{r_i}) < P(F,\phi)$. 
(To see this we follow the argument in section 7.23 of \cite{Ruelle}. First note that
$\theta_i(\Sigma_i^{r_i})$ is a closed $\phi$-invariant proper subset of $M$.
Since $\theta_i$ is bounded-to-one, $P(F \circ \theta_i,\sigma_{r_i}) \le P(F|_{\theta_i(\Sigma_i^{r_i})},\phi)$. 
But since the equilibrium state of $F$ is fully supported, it is easy to see by the variational principle that
$P(F|_{\theta_i(\Sigma_i^{r_i})},\phi) < P(F,\phi)$.)
In particular, the suspended flows $\sigma_{r_i}^t$ have topological entropy 
strictly less than $h$. 

For $F : M \to \mathbb R$, write
\[
N_{\phi}(T,F)=
\sum_{\substack{\gamma\in \mathcal{P}(\phi):\\ l(\gamma)\le T}} \exp\left({\int_\gamma F}\right)
\]
(with similar definitions for the other flows). Then we have the following corollary.

\begin{corollary} \label{cor-bowenmanning}
For every H\"older continuous function $F : M \to \mathbb R$, we have 
$$
N_{\phi}(T,F)  = N_{\sigma_r}(T,F \circ \theta) +  O(e^{h_F^\prime T}),
$$
for some $h_F' < P(F,\phi)$. 
In particular,
\[
N_{\phi}(T,\langle \xi,\Psi \rangle)  = N_{\sigma_r}(T,\langle \xi, \Psi \circ \theta\rangle) +  O(e^{h^\prime T}),
\]
for some $h^\prime < P(\langle \xi,\Psi \rangle,\phi)$.
\end{corollary}

\section{Covers and Group Extensions} \label{coversandgroupextensions}
The aim of this section is to interpret the quantities $h$, $h(X)$ and $h(X^{\mathrm{ab}})$
in terms of a subshift of finite type and its group extensions.

Consider a regular cover $p_X : X \to M$ of $M$, with covering group $G$
(acting as isometries on $X$), and the lifted flow $\phi_X^t : X \to X$.
We will now describe the procedure to ``lift'' the symbolic dynamics for 
$\phi^t:M\to M$ to $\phi_X^t:X\to X$.
We use the decoration $\mathcal{R}^M=\left\{R^M_1,\cdots, R^M_k\right\}$ for the Markov section for $\phi^t:M\to M$. For every pair $(i,j)$ such that $A(i,j)=1$, there is a minimum $t_{ij}>0$ with 
$\phi^{t_{ij}}(R^M_i)\cap R^M_j \ne\varnothing$. 
These $t_{ij}$ are clearly independent of the covering $X$ and $\mathcal R^M$ may be chosen so that $t_{ij}$
is arbitrarily small.

We assume that the rectangles in $\mathcal{R}^M$ have sufficiently small diameter so that 
they are each contained  in an open ball that is simply connected. For each $R^M\in \mathcal{R}^M$, 
fix a lift $R^X\subseteq X$. (Or more precisely, fix a connected open set $U^M$ containing $R^M$, 
and a connected lift $U^X$. Then $R^X$ is chosen to lie inside $U^X$.) 
Having chosen the flow time between rectangles, and diameter of rectangles 
to be sufficiently small, we deduce that there is a unique $g = g(i,j)\in G$ with 
$\phi^{t_{ij}}(R^X_i)\cap g R^X_j \ne \varnothing$.
We now define a function $\psi:\Sigma\to G$ by 
$\psi(x) = g(x_0,x_1)$.

We will use the function $\psi$ to identify periodic orbits in $M$ that lift to periodic obits in $X$. 
We precede this discussion by elaborating on the \emph{$G$-class} of a closed curve. 
Let $x \in M$ 
be fixed.
Given that the covering map $p:X\to M$ is fixed, we have an isomorphism $G=\pi_1(M,x)/N$, where $N$ is a normal
subgroup of $\pi_1(M,x)$ (isomorphic to $\pi_1(X)$). 

For an arbitrary closed curve $\gamma:[0,R]\to M$, we can choose paths connecting $x$ to $\gamma(0)$ 
to produce a closed curve based at $x$ and hence a homotopy class in $\pi_1(M,x)$. 
A different choice of 
path from $x$ to $\gamma(0)$ gives a 
conjugate element of $\pi_1(M,x)$. Thus we obtain a well-defined conjugacy class in $G$,
which we will denote by $\langle \gamma \rangle_X$ and call the $G$-{\it class} of
$\gamma$. 
We say that the $G$-class is {\it trivial} if $\langle \gamma \rangle_X$ consists only of the identity,
in which case we slightly abuse notation by writing $\langle \gamma \rangle_X = e$,
and this is equivalent to $\gamma$ lifting to a closed curve in $X$
Clearly, a closed curve $\gamma$ has 
trivial $G$-class if any choice of path $x$ to $\gamma(0)$ gives a closed curve with homotopy 
class in $N$.

We claim that $\psi$ codes orbits with trivial $G$-class, but we can only make statements 
about points for which the semi-conjugacy $\theta$ is actually a bijection. 
Write $\mathcal{P}(\phi)^*$ for the set of periodic orbits in $M$ for which $\theta$ induces a bijection.

\begin{proposition}\label{skew1}
Suppose that $\{x,\sigma x,\ldots,\sigma^{n-1}x\}$ is a periodic $\sigma$-orbit (with $\sigma^n x=x$)
corresponding to $\gamma \in \mathcal P(\phi)^*$. Then
$\gamma$ has trivial $G$-class if and only if 
$\psi^n(x)
= e$.
\end{proposition} 

\begin{proof}
We can regard $\theta(x,0)$ as the initial point on $\gamma$. Let $\widetilde \gamma$ be a lift of 
$\gamma$, we we may assume starts in $R_{x_0}^X$. By construction, $\widetilde \gamma$ ends 
in $\psi^n(x)$ and is thus periodic if and only if $\psi^n(x)=e$.
\end{proof}

Now suppose that $Y=X^{\mathrm{ab}}$ is the maximal abelian subcover of $X$, with abelian covering group $G^{\mathrm{ab}}=G/[G,G]$. Then $\langle \gamma \rangle_Y = [\gamma]$.
We define $\psi_Y:\Sigma\to G^{\mathrm{ab}}$ 
in an analogous way to $\psi : \Sigma \to G$, i.e. $\psi_Y(x)$ is given by
$\phi^{t_{x_0x_1}}(R^Y_{x_0})\cap \psi_Y(x)R^Y_{x_1} \ne \varnothing$, where $R^Y_i$
is the projection of $R_i^X$ to $Y$. Clearly, $\psi_Y$ is the composition of $\psi$ and the projection 
homomorphism from $G$ to $G^{\mathrm{ab}}$.
We immediately have the following.

\begin{proposition}\label{skew2}
Suppose that $\{x,\sigma x,\ldots,\sigma^{n-1}x\}$ is a periodic $\sigma$-orbit (with $\sigma^n x=x$)
corresponding to $\gamma \in \mathcal P(\phi)^*$. Then
$[\gamma] = \psi_Y^n(x)$ (and, in particular, $[\gamma] =0$ if and only if 
$\psi_Y^n(x)=0$).
\end{proposition} 

\begin{proof}
This follows from Proposition \ref{skew1} and the fact that $G^{\mathrm{ab}}$ is abelian.
\end{proof}

We now wish to make two reductions. First of all, if $G^{\mathrm{ab}} \cong \mathbb Z^a \times G_0$ 
with $G_0$ a non-trivial finite abelian group then
$Y$ is a finite cover of a manifold $Y'$, which is a $\mathbb Z^a$-cover of $M$. 
By Lemma \ref{zeromean}, $h(Y)=h(Y')$ and so there is no loss of generality in replacing
$Y$ with $Y'$ is our analysis. 
Slightly abusing notation, we shall still use $Y$ to denote this $\mathbb Z^a$-cover of $M$. 
Then $\langle \gamma \rangle_Y = [\gamma] = \int_\gamma \Psi$ is an element of $\mathbb Z^a$.

Secondly, wish to switch from the two-sided subshift of finite type $\sigma : \Sigma \to \Sigma$ to the corresponding one-sided subshift of finite type
$\sigma : \Sigma^+ \to \Sigma^+$.
There is an obvious identification between the periodic orbits of the two systems.
Since the functions $\psi$ and $\psi_Y$ depend on only two co-ordinates, we can equally well regard them as functions defined on the one-sided shift space $\Sigma^+$. Also, by Lemma \ref{sinai}, we may change $r$ by
the addition of a coboundary to obtain a function defined on $\Sigma^+$. We will continue to denote this modified function by $r$ and note that its sums around periodic orbits are unchanged.

After these two reductions, we have extended systems
$T_\psi : \Sigma^+ \times G \to \Sigma^+ \times G$ and $T_{\psi_Y} : \Sigma \times \mathbb Z^a
\to \Sigma^+ \times \mathbb Z^a$.

We end the section with two results on transitivity.

\begin{lemma} \label{transitiveimpliestransitive}
If $\phi_X^t :X\to X$ is transitive, then $T_\psi$ is transitive.
\end{lemma}

\begin{proof}
Let $x \in X$ be a point with dense $\phi_X$-orbit. Without loss of generality
$x \in \bigcup_{i=1}^k \bigcup_{g \in G} g \cdot R_i^X$ and then
$\{P_X^nx\}_{n=-\infty}^\infty$ is dense in $\bigcup_{i=1}^k \bigcup_{g \in G} g \cdot R_i^X$,
where $P_X$ is the Poincar\'e map between the sections in $X$.
Suppose that 
$\widetilde A((i_j,g_j),(i_{j+1},g_{j+1}))=1$, where $\widetilde A$ is the transition
matrix for $\Sigma^+ \times G$, for
$j=0,\ldots,n$. Then
\[
U  = \bigcap_{j=0}^n P_X^{-j}(\mathrm{int}(g_j \cdot  R_{i_j}^X))
\]
is non-empty and open in $\bigcup_{i=1}^k \bigcup_{g \in G} g \cdot R_i^X$. 
(Here $\mathrm{int}(g_j \cdot R_{i_j}^X)$ is taken with respect to the
co-dimension one disk containing $g_j \cdot R_{i_j}^X$.)
Since $x$ has dense $P_X$-orbit, $P_X^m x \in U$ for some
$m \in \mathbb Z$. 
Then $P_X^{m+j}(x) \in \mathrm{int}(g_j \cdot R_{i_j}^X)$
for $j=0,\ldots,n$. By definition, this implies that the $T_\psi$-orbit of 
$(\theta(p_X(x)),g_0) \in \Sigma^+ \times G$ 
(where $\theta(p_X(x))$ is identified with a point in the one-sided shift)
passes through the (arbitrary) cylinder  
$[(i_0,g_0),\ldots,(i_n,g_n)]$ and is thus dense in $\Sigma^+ \times G$. Therefore,
$\widetilde \sigma : \Sigma^+ \times G \to \Sigma^+ \times G$ is transitive. 
\end{proof}

\begin{lemma}\label{transitiveimpliesfull}
If $\phi_Y^t : Y \to Y$ is transitive then the Anosov flow $\phi^t : M \to M$ is $Y$-full. 
 \end{lemma}

\begin{proof}
Suppose that $\phi_Y$ is transitive. Then, by Lemma \ref{transitiveimpliestransitive}, $T_{\psi_Y} : \Sigma^+ \times \mathbb Z^a \to \Sigma^+ \times \mathbb Z^a$, is transitive.
In addition, 
\[
\bigcup_{n=1}^\infty \{\psi_Y^n(x) \hbox{ : } \sigma^nx=x\} = \{[\gamma] \hbox{ : } \gamma \in \mathcal P(\phi)\},
\]

Suppose $\phi$ is {\it not} $Y$-full. Then, by Proposition \ref{asymptotics_abeliancover}, 
$\bigcup_{n=1}^\infty \{\psi_Y^n(x) \hbox{ : } \sigma^nx=x\}$ is contained in a closed half-space in 
$\mathbb R^a$
and hence there exists $t \in \mathbb R^a \setminus \{0\}$ such that 
$\langle t , \psi_Y^n(x)\rangle \ge 0$, whenever $\sigma^n x=x$. 
Since measures supported on periodoc orbits are dense in the space of $\sigma$-invariant probability
measaures, this gives that 
$\int \langle t , \psi_Y\rangle \, dm \ge 0$ for every $\sigma$-invariant probability measure $m$.
By a result of Savchenko \cite{Savchenko}, this implies that
there exists a continuous function $u : \Sigma \to \mathbb R$ such that
$\langle t, \psi_Y\rangle + u \circ \sigma - u \ge 0$. Hence, for all $n \ge 1$,
$\langle t , \psi_Y^n \rangle + u \circ \sigma^n - u \ge 0$, giving
$\langle t , \psi_Y^n \rangle \ge -2\|u\|_\infty$.
Since $T_{\psi_Y}^n(x,k) = (\sigma^n x,k+\psi_Y^n(x))$, if we choose $\alpha \in \mathrm Z^a$ with 
$\langle t , \alpha \rangle < -2\|u\|_\infty$, then no $T_{\psi_Y}$-orbit starting in
 $\Sigma \times \{0\}$ can reach $\Sigma \times \{\alpha\}$, and so $T_{\psi_Y}$ is not transitive.
\end{proof}

\section{Proof of Theorem \ref{intro-main}} \label{proofofmaintheorem}

We begin with a lemma relating the quantities $h(X)$ and $h(Y)$ to Gurevi\v{c}
pressures for $T_\psi$ and $T_{\psi_Y}$. 
Let $\sigma : \Sigma ^+ \to \Sigma^+$ be the one-sided subshift of finite type and 
let $r : \Sigma^+ \to \mathbb R$ (cohomologous to a strictly positive function),
 $\psi=\psi_X : \Sigma^+ \to G$ and $\psi_Y : \Sigma^+ \to \mathbb Z^a$
be the functions constructed in the previous section.

We have the following key lemma.

\begin{lemma}\label{coding}
The following three 
statements hold:
\begin{enumerate}
\item
$h=P(0,\phi)$ is the unique zero of $s\mapsto P(-sr,\sigma)$; 
\item
$h(Y)=P(\langle \xi, \Psi \rangle,\phi)$, and $h(Y)$ is the unique 
zero of $s\mapsto P_{\mathrm{Gur}}(-sr+\langle \xi,\Psi \rangle,T_{\psi_Y})$; 
\item
$h(X)=h(Y)$ if and only if 
$P_{\mathrm{Gur}}(-h(Y)r+\langle \xi,\psi_Y \rangle,T_\psi)=0$.
\end{enumerate}
\end{lemma}

\begin{proof}
Part (1) is a standard result, see \cite{PP}.

Let us now prove (2). By Lemma \ref{zeromean}, we know that 
$h(Y)=P(\langle \xi, \Psi \rangle,\phi)$. 
Also, by Lemma \ref{sameabscissa}, $h(Y)$ is the abscissa of convergence of 
\[
\sum_{\substack{\gamma \in \mathcal P(\phi) : \\ [\gamma]=0}}
 \frac{\Lambda(\gamma)}{l(\gamma)} e^{-sl(\gamma)}
 =\sum_{\substack{\gamma \in \mathcal P(\phi) : \\ [\gamma]=0}}
 \frac{\Lambda(\gamma)}{l(\gamma)} e^{-sl(\gamma)+ \langle \xi,[\gamma] \rangle}
 \]
 By Corollary \ref{cor-bowenmanning}, 
 \[
 \sum_{\substack{\gamma \in \mathcal P(\phi) : \\ [\gamma]=0}}
 \frac{\Lambda(\gamma)}{l(\gamma)} e^{-sl(\gamma)+ \langle \xi,[\gamma] \rangle}
 \]
 differs from
 \[
\sum_{n=1}^\infty \frac{1}{n} 
\sum_{\substack{\sigma^n x=x : \\ \psi_Y^n(x)=0}}
e^{-sr^n(x)}
\]
by a series with abscissa of convergence $h'< P(\langle \xi,\Psi\rangle ,\phi) = h(Y)$.
Therefore 
$P_{\mathrm{Gur}}(-h(Y)r+\langle \xi,\Psi \rangle,T_{\psi_Y})=0$.
Since $r$ is strictly positive, Lemma \ref{strictlydecreasing} gives that
$s \mapsto P_{\mathrm{Gur}}(-sr,T_{\psi_Y})$ is strictly decreasing and therefore $h(Y)$ is the unique zero.

Finally, we prove (3). 
Write
\[
S_1(s) := \sum_{\substack{\gamma \in \mathcal P(\phi) : \\ \langle\gamma\rangle=0}}
 e^{-sl(\gamma)}
\quad \mathrm{and} \quad
S_2(s) :=\sum_{n=1}^\infty \frac{1}{n} 
\sum_{\substack{\sigma^n x=x : \\ \psi^n(x)=e}}
e^{-sr^n(x)}.
\]
We have $h(Y)=h(X)$ if and only if
$h(Y)$ is the abscissa of convergence of 
$S_1(s)$.
Since $\langle \gamma \rangle = e$ implies that $[\gamma]=0$, we can use 
argument in the proof of Lemma \ref{sameabscissa} 
and Corollary \ref{cor-bowenmanning} to show that 
the abscissa of convergence of $S_1(s)-S_2(s)$
is bounded above by $\max\{h(Y)/2,h'\} < h(Y)$. 
Hence, $S_1(s)$ has abscissa of convergence $h(Y)$ if and only if
$S_2(s)$ has abscissa of convergence $h(Y)$, which in turn is equivalent to 
$P_{\mathrm{Gur}}(-h(Y)r ,T_\psi)=0$.
But, since $\psi^n(x)=e$ implies that 
$\psi_Y^n(x) =0$, we can employ the argument used to prove (2) 
to show that, for any $s \in \mathbb R$,
\[
P_{\mathrm{Gur}}(-sr +\langle \xi,\psi_Y \rangle,T_\psi)
=P_{\mathrm{Gur}}(-sr ,T_\psi),
\]
completing the proof.
\end{proof}

We now complete the proof of Theorem \ref{intro-main}.

\begin{proof}[Proof of Theorem \ref{intro-main}]
We first deal with the case $a>0$.
Suppose that $G$ is non-amenable.
By Theorem 1.1 of \cite{Jaerisch}, since $G$ is non-amenable, we have that
\[
\log \mathrm{spr}_{\mathcal H}(\mathcal{L}_{-h(Y)r+\langle \xi,\psi_Y \rangle})
<P(-h(Y)r+\langle \xi,\psi_Y \rangle,\sigma).
\]
Moreover, by the same theorem, for any H\"older continuous $f : \Sigma^+ \to \mathbb R$, we always have
$P_{\mathrm{Gur}}(f,T_\psi) \le \log 
\mathrm{spr}_{\mathcal H}(\mathcal{L}_f)$.
Putting these together gives
\begin{align*}
P_{\mathrm{Gur}}(-h(Y)r+\langle \xi,\psi_Y \rangle,T_\psi)
&<
P(-h(Y)r+\langle \xi,\psi_Y \rangle,\sigma) \\
&=P(-h(Y)r + \langle\xi,\psi_Y \rangle,T_{\psi_Y}),
\end{align*}
where the last identity follows from Lemma \ref{coding}.
Now, by the proof of part (3) of Lemma \ref{coding}, this becomes
$P_{\mathrm{Gur}}(-h(Y)r,T_\psi) <0$ and therefore
$h(X)<h(Y)$.

Now suppose that $G$ is amenable.
By Theorem \ref{intro-main2}, 
$P_{\mathrm{Gur}}(-sr+\langle \xi,\psi^{\mathrm{ab}}\rangle,T_\psi)=
P_{\mathrm{Gur}}(-sr+\langle \xi,\psi^{\mathrm{ab}}\rangle,T_{\psi^{\mathrm{ab}}})$ for any $s$. 
By, Lemma \ref{coding}, this implies that $h(X)=h(Y)$.

Finally, we deal with the case where $a=0$, which requires a simple modification of the above arguments. 
In this case, $G^{\mathrm{ab}}$ is finite and so 
$h(Y)=h$. If $G$ is non-amenable then
\[
P_{\mathrm{Gur}}(-hr,T_\psi) 
\le \log \mathrm{spr}_{\mathcal H}(\mathcal{L}_{-hr})
<P(-hr,\sigma) =0
\]
so $h(X)<h$.
If $G$ is amenable then Theorem \ref{intro-main2} gives
\[
P_{\mathrm{Gur}}(-sr,T_\psi)=
P_{\mathrm{Gur}}(-sr,T_{\psi^{\mathrm{ab}}}) =P(-sr,\sigma),
\]
so $h(X)=h$.
\end{proof}

\section{Equidistribution} \label{sectiononequidistribution}

In this final section, we consider the spatial distribution of $\phi$-periodic orbits with trivial $G$-class, in the case where $G$ is amenable. This generalises Theorem 2 of \cite{Sharp93} which deals with $G$ abelian.
Write
\[
\Xi(T,\epsilon) = \#\{\gamma \in \mathcal P(\phi) \hbox{ : } T<l(\gamma) \le T+\epsilon, \ \langle \gamma \rangle =e\}.
\]
Our results will be based on the following definition.

\begin{definition}
We say the $\phi$-periodic orbits with trivial $G$-class are equidistributed with respect to a measure $\mu$ 
if, for every $\epsilon>0$ and 
continuous function $F : M \to \mathbb R$, we have
\[
\lim_{T \to \infty}
\frac{1}{\Xi(T,\epsilon)}
\sum_{\substack{\gamma \in \mathcal P(\phi) : \\
T<l(\gamma) \le T+\epsilon, \ \langle \gamma \rangle =e}}
\frac{1}{l(\gamma)} \int_\gamma F 
= \int F \, d\mu.
\]
\end{definition}

We have the following equidistribution result for Anosov flows.

\begin{theorem} \label{equidistribution}
Let $\phi^t : M \to M$ be an Anosov flow and let $X$ be a regular cover of $M$ with amenable 
covering group $G$. Suppose that the lifted flow $\phi_X^t : X \to X$ is transitive. Then
 the $\phi$-periodic orbits with trivial $G$-class are equidistributed with respect to 
 $\mu_{\langle \xi,\Psi\rangle}$, 
the equilibrium state for $\langle \xi,\Psi \rangle$.
\end{theorem}

Specialising to geodesic flows gives the following corollaries.

\begin{corollary}
Let $\phi^t : SV \to SV$ be the geodesic flow on the unit-tangent bundle over a compact negatively curved manifold $V$
and let $X$ be a regular cover of $M$ with amenable 
covering group $G$. Then
 the $\phi$-periodic orbits with trivial $G$-class are equidistributed with respect to
 the measure of maximal entropy for $\phi$.
\end{corollary}

\begin{proof}
Since $\phi$ is a geodesic flow, $\xi=0$. Also, since $G$ is amenable, $G$ is not equal to $\pi_1(M)$ and 
hence,
by a result of Eberlein \cite{Eberlein}, the lifted flow $\phi_X^t : X \to X$ is transitive.
Thus the corollary follows from Theorem \ref{equidistribution}.
\end{proof}

\begin{corollary}
Let $V$ be a compact manifold with negative sectional curvatures and let $\widehat V$ be a regular cover 
with amenable covering group $G$. Then the closed geodesics on $V$ with trivial $G$-class are 
are equidistributed with respect to the projection to $V$ of the measure of maximal entropy for the geodesic
flow on the unit-tangent bundle $SV$.
\end{corollary}

If $V$ has constant negative curvature then the measure of maximal entropy is equal to the volume measure 
and so we haver the following.

\begin{corollary}
Let $V$ be a compact hyperbolic manifold and let $\widehat V$ be a regular cover 
with amenable covering group $G$. Then the closed geodesics on $V$ with trivial $G$-class are 
are equidistributed with respect to the volume measure on $V$.
\end{corollary}

Theorem \ref{equidistribution} is a consequence of the following large deviations result.
For $\gamma \in \mathcal P(\phi)$, let $\mu_\gamma$ denote the $\phi$-invariant probability 
measure defined by
\[
\int F \, d\mu_\gamma = \frac{1}{l(\gamma)}\int_\gamma F.
\]

\begin{theorem} \label{ld}
Let $\phi^t : M \to M$ be an Anosov flow and let $X$ be a regular cover of $M$ with amenable 
covering group $G$. Suppose that the lifted flow $\phi_X^t : X \to X$ is transitive. Then, for any 
weak$^*$-compact set $\mathcal K \subset \mathcal M(\phi)$ such that 
$\mu_{\langle \xi,\Psi\rangle}
\notin \mathcal K$, we have
\[
\limsup_{T \to \infty}
\frac{1}{T} \log
\left(
\frac{
\#\{\gamma \in \mathcal P(\phi) \hbox{ : } T < l(\gamma) \le T+ \epsilon, \ 
\langle \gamma \rangle =e, \ \mu_\gamma \in \mathcal K\}
}
{
\#\{\gamma \in \mathcal P(\phi) \hbox{ : } T < l(\gamma) \le T+ \epsilon, \ 
\langle \gamma \rangle =e\}
} 
\right)<0.
\]
\end{theorem}

\begin{proof}
To shorten some formulae, we will write 
\[
Q(F) := P(\langle \xi,\Psi \rangle+F,\phi).
\]
We begin by noting the following. By Theorem \ref{intro-main}, 
\begin{equation} \label{ld-exactgrowthrate}
\lim_{T \to \infty}
\frac{1}{T} \log
\#\{\gamma \in \mathcal P(\phi) \hbox{ : } T < l(\gamma) \le T+ \epsilon, \ 
\langle \gamma \rangle =e\}
= Q(0)
\end{equation}
and, for any continuous function $F : M \to \mathbb R$,
\begin{align*}
\sum_{\substack{\gamma \in \mathcal P(\phi) \\
T < l(\gamma) \le T+\epsilon, \, \langle \gamma \rangle =e}}
\exp\left(\int_\gamma F\right) 
&=
\sum_{\substack{\gamma \in \mathcal P(\phi) \\
T < l(\gamma) \le T+\epsilon, \, \langle \gamma \rangle =e}}
\exp\left(\int_\gamma \langle \xi,\Psi \rangle + F\right) 
\\
&\le 
\sum_{\substack{\gamma \in \mathcal P(\phi) \\
T < l(\gamma) \le T+\epsilon}}
\exp\left(\int_\gamma \langle \xi,\Psi \rangle + F\right), 
\end{align*}
giving 
\begin{equation} \label{ld-upperboundongrowthrate}
\lim_{T \to \infty}
\frac{1}{T} \log
\sum_{\substack{\gamma \in \mathcal P(\phi) \\
T < l(\gamma) \le T+\epsilon, \, \langle \gamma \rangle =e}}
\exp\left(\int_\gamma F\right) 
\le Q(F).
\end{equation}
 
Now define 
\[
\rho = \rho_{\mathcal K} := \inf_{\nu \in \mathcal K} \sup_{F \in C(M,\mathbb R)} 
\left(\int F \, d\nu - Q(F)\right).
\]
Fix $\delta>0$. For every $\nu \in \mathcal K$, there exists $F \in C(M,\mathbb R)$ such that 
\[
\int F \, d\nu - Q(F) > \rho -\delta.
\]
Thus we have
\[
\mathcal K \subset
\bigcup_{F \in C(M,\mathbb R)} \left\{\nu \in \mathcal M(\phi)
\hbox{ : } \int F \, d\nu - Q(F) > \rho-\delta\right\}.
\]
Since $\mathcal K$ is weak$^*$-compact, we can find a finite subcover
\[
\mathcal K \subset
\bigcup_{i=1}^k \left\{\nu \in \mathcal M(\phi)
\hbox{ : } \int F_i \, d\nu - Q(F) > \rho-\delta\right\}.
\]
This gives us the inequality
\begin{align*}
&\#\{\gamma \in \mathcal P(\phi) \hbox{ : }
T < l(\gamma) \le T+\epsilon, \, \langle \gamma \rangle =e, \, \mu_\gamma \in \mathcal K\} \\
&\le
\sum_{i=1}^k \#\left\{\gamma \in \mathcal P(\phi) \hbox{ : }
T < l(\gamma) \le T+\epsilon, \, \langle \gamma \rangle =e, \, \int F_i \, d\mu_\gamma -Q(F) > \rho-\delta\right\}
\\
&\le
\sum_{i=1}^k 
\sum_{\substack{\gamma \in \mathcal P(\phi) \\
T < l(\gamma) \le T+\epsilon, \, \langle \gamma \rangle =e}}
 \exp\left(-l(\gamma)(Q(F_i)+(\rho-\delta))+\int_\gamma F_i\right)
\\
&\le
\sum_{i=1}^k \max\{e^{-T(Q(F_i)+(\rho-\delta))}, e^{-(T+\epsilon)(Q(F_i)+(\rho-\delta))}\}
\sum_{\substack{\gamma \in \mathcal P(\phi) \\
T < l(\gamma) \le T+\epsilon, \, \langle \gamma \rangle =e}}
\exp\left(\int_\gamma F_i\right)
\end{align*}
and hence, using (\ref{ld-upperboundongrowthrate}) we obtain the growth rate estimate
\begin{align*}
&\limsup_{T \to \infty} \frac{1}{T} \log
\#\{\gamma \in \mathcal P(\phi) \hbox{ : }
T < l(\gamma) \le T+\epsilon, \, \langle \gamma \rangle =e, \, \mu_\gamma \in \mathcal K\}
\le -\rho + \delta.
\end{align*}
Since $\delta>0$ is arbitrary, we combine this with (\ref{ld-exactgrowthrate}) to deduce that
\[
\limsup_{T \to \infty}
\frac{1}{T} \log
\left(
\frac{
\#\{\gamma \in \mathcal P(\phi) \hbox{ : } T < l(\gamma) \le T+ \epsilon, \ 
\langle \gamma \rangle =e, \ \mu_\gamma \in \mathcal K\}
}
{
\#\{\gamma \in \mathcal P(\phi) \hbox{ : } T < l(\gamma) \le T+ \epsilon, \ 
\langle \gamma \rangle =e\}
} 
\right) \le -\rho -Q(0).
\]

To complete the proof, we will show that $\rho+Q(0)>0$. First note that if 
$\nu \neq \mu_{\langle \xi.\Psi\rangle}$ then
\begin{align*}
&\sup_{F \in C(M,\mathbb R)} 
\left(\int F \, d\nu -Q(F) +Q(0)\right) \\
&= \sup_{F \in C(M,\mathbb R)} 
\left(\int F \, d\nu -P(\langle \xi,\Psi \rangle+F,\phi) +P(\langle \xi,\Psi \rangle,\phi)\right) \\
&= \sup_{F \in C(M,\mathbb R)} 
\left(\int (F - \langle \xi,\Psi\rangle) \, d\nu -P(F,\phi) + P(\langle \xi,\Psi \rangle,\phi)\right) \\
&= \sup_{F \in C(M,\mathbb R)}
\left(\int F \, d\nu -P(F,\phi)\right) + P(\langle \xi,\Psi \rangle,\phi) - \int \langle \xi,\Psi \rangle \, d\nu \\
&= -\inf_{F \in C(M,\mathbb R)}
\left(P(F,\phi)- \int F \, d\nu\right) + P(\langle \xi,\Psi \rangle,\phi) - \int \langle \xi,\Psi \rangle \, d\nu.
\end{align*}
By Lemma \ref{entropyusc}, this is equal to
\[
-h_\phi(\nu) + P(\langle \xi,\Psi \rangle,\phi) - \int \langle \xi,\Psi \rangle \, d\nu >0,
\]
where the last inequality comes from the uniqueness of the equilibrium state $\mu_{\langle \xi,\Psi\rangle}$.
Also by Lemma \ref{entropyusc}, the map on 
$\mathcal M(\phi)$  given by 
\[
\nu \mapsto 
-h_\phi(\nu) + P(\langle \xi,\Psi \rangle,\phi) - \int \langle \xi,\Psi \rangle \, d\nu
\]
is lower semi-continuous and so we can conclude that $\rho+Q(0)>0$, as required.
\end{proof}

\begin{proof}[Proof of Theorem \ref{equidistribution}]
Given a function $F \in C(M,\mathbb R)$,
choose $\delta>0$ and set $\mathcal K$ to be the compact set
\[
\mathcal K = \left\{\nu \in \mathcal M(\phi) \hbox{ : }
\left|\int F \, d\nu - \int F \, d\mu_{\langle \xi,\Psi \rangle}\right| \ge \delta\right\}.
\] 
Applying Theorem \ref{ld}, we have
\begin{align*}
&\frac{1}{\Xi(T,\epsilon)}
\sum_{\substack{\gamma \in \mathcal P(\phi) : \\
T<l(\gamma) \le T+\epsilon \\ \langle \gamma \rangle =e}}
 \int F \, d\mu_\gamma
=
\frac{1}{\Xi(T,\epsilon)}
\sum_{\substack{\gamma \in \mathcal P(\phi) : \\
T<l(\gamma) \le T+\epsilon \\ \langle \gamma \rangle =e, \ \mu_\gamma \notin \mathcal K}}
\int F \, d\mu_\gamma +O(e^{-\eta T}),
\end{align*}
for any $0 <\eta < \rho_{\mathcal K}  + P(\langle \xi,\Psi \rangle,\phi)$, so we only need to consider the limit of the first term on the Right Hand Side. 
We have
\begin{align*}
\frac{1}{\Xi(T,\epsilon)}
\sum_{\substack{\gamma \in \mathcal P(\phi) : \\
T<l(\gamma) \le T+\epsilon \\ \langle \gamma \rangle =e, \ \mu_\gamma \notin \mathcal K}}
\int F \, d\mu_\gamma
&=
(1-O(e^{-\eta T}))
\int F \, d\mu_{\langle \xi,\Psi \rangle}
\\
&+
\frac{1}{\Xi(T,\epsilon)}
\sum_{\substack{\gamma \in \mathcal P(\phi) : \\
T <l(\gamma) \le T+\epsilon \\ \langle \gamma \rangle =e, \ \mu_\gamma \notin \mathcal K}}
\left( \int F \, d\mu_\gamma - \int F \, d\mu_{\langle \xi,\Psi \rangle}\right)
\end{align*}
and therefore
\[
\limsup_{T \to \infty} \frac{1}{\Xi(T,\epsilon)}
\sum_{\substack{\gamma \in \mathcal P(\phi) : \\
T<l(\gamma) \le T+\epsilon \\ \langle \gamma \rangle =e, \ \mu_\gamma \notin \mathcal K}}
\int F \, d\mu_\gamma
\le \int F \, d\mu_{\langle \xi,\Psi \rangle} + \delta
\]
and 
\[
\liminf_{T \to \infty} \frac{1}{\Xi(T,\epsilon)}
\sum_{\substack{\gamma \in \mathcal P(\phi) : \\
T<l(\gamma) \le T+\epsilon \\ \langle \gamma \rangle =e, \ \mu_\gamma \notin \mathcal K}}
\int F \, d\mu_\gamma
\ge \int F \, d\mu_{\langle \xi,\Psi \rangle} - \delta.
\]
Since $\delta>0$ is arbitrary, this completes the proof.
\end{proof}

\end{document}